\documentclass[12pt,a4paper]{amsart}
\usepackage{amssymb,amsthm,amsmath,currfile}
\usepackage{geometry}
\usepackage[english]{babel}
\usepackage{verbatim,here}
\usepackage[T1]{fontenc}
\usepackage{graphicx}
\usepackage{epstopdf}
\usepackage{floatflt}
\usepackage{color,xcolor}
\usepackage{enumerate}
\usepackage{bm}
\usepackage{a4wide}

\font\fFt=eusm10
\font\fFa=eusm7
\font\fFp=eusm5
\def\K{\mathchoice
{\hbox{\,\fFt K}} {\hbox{\,\fFt K}} {\hbox{\,\fFa K}}
{\hbox{\,\fFp K}}}

\font\fFt=eusm10
\font\fFa=eusm7
\font\fFp=eusm5
\def\E{\mathchoice
{\hbox{\,\fFt E}} {\hbox{\,\fFt E}} {\hbox{\,\fFa E}}
{\hbox{\,\fFp E}}}

\makeatletter
\renewcommand{\subsection}{
    \stepcounter{subsection}
    \addtocounter{equation}{+1}
    \setcounter{subsection}{\value{equation}}
    \bigskip
    \noindent{{\bfseries \arabic{section}.\arabic{subsection}.\ }}}
\renewcommand{\thesubsubsection}{
    \@arabic\c@section.\@arabic\c@subsection.\@arabic\c@subsubsection}
\makeatother

\newcounter{minutes}
\divide\time by 60
\newcounter{hours}
\multiply\time by 60 \addtocounter{minutes}{-\time}

\curraddr{} \email{semen.nasyrov@yandex.ru} \curraddr{}

\curraddr{} \email{sugawa@math.is.tohoku.ac.jp}

\curraddr{} \email{vuorinen@utu.fi}


\dedicatory{}
\commby{}

\theoremstyle{definition}
\newtheorem{rem}[equation]{Remark}
\newtheorem{example}[equation]{Example}

\theoremstyle{plain}
\newtheorem{thm}[equation]{Theorem}
\newtheorem{cor}[equation]{Corollary}
\newtheorem{lem}[equation]{Lemma}

\newtheorem{conjecture}[equation]{Conjecture}

\theoremstyle{definition}

\theoremstyle{remark}

\newtheorem{mysubsection}[equation]{}

\numberwithin{equation}{section}

\newcommand{\beq}{\begin{equation}}
\newcommand{\eeq}{\end{equation}}
\newcommand{\ben}{\begin{enumerate}}
\newcommand{\een}{\end{enumerate}}
\newcommand{\bequu}{\begin{eqnarray*}}
\newcommand{\eequu}{\end{eqnarray*}}
\newcommand{\bequ}{\begin{eqnarray}}
\newcommand{\eequ}{\end{eqnarray}}

\renewcommand{\Im}{{ \rm Im}\,}
\renewcommand{\Re}{{ \rm Re}\,}

\renewcommand{\thefootnote}{\number_style{footnote}}
\begin{document}
\def\thefootnote{}

\title[]{Moduli of quadrilaterals and quasiconformal reflection}
\author[S.~Nasyrov]{Semen Nasyrov}
\address{Institute of Mathematics and Mechanics. Kazan Federal University, Kazan, Russia }

\author[T.~Sugawa]{Toshiyuki Sugawa}
\address{Graduate School of Information Sciences,
Tohoku University, Aoba-ku, Sendai 980-8579, Japan}

\author[M.~Vuorinen]{Matti Vuorinen}
\address{Department of Mathematics and Statistics, University of Turku,
         Turku, Finland}
\date{}

\begin{abstract}
We study the interior and exterior moduli of polygonal quadrilaterals.
The main result is a formula for a conformal mapping of the upper
half plane onto the exterior of a convex polygonal quadrilateral. We prove this
by a careful analysis of the Schwarz-Christoffel transformation and obtain
the so-called accessory parameters and then the result
in terms of the  Lauricella hypergeometric function. This result enables us to understand
the dissimilarities of the exterior and interior of a convex polygonal quadrilateral.
We also give a Mathematica algorithm for the computation.
In particular, we study
the special case of an isosceles trapezoidal polygon $L$
and obtain some estimates for the coefficient of quasiconformal
reflection over $L$ in terms of special functions and
geometric parameters of~$L$.
\end{abstract}

\maketitle

\makeatother

\section{Introduction}

\label{section1}

\setcounter{equation}{0}

A quadrilateral $\mbox{\boldmath$Q$}=(Q;z_1,z_2,z_3,z_4)$ is a
Jordan domain $Q$ on the Riemann sphere with four fixed points
$z_1$, $z_2$, $z_3$, and $z_4$ on its boundary. We label $z_j$
in such an order that increasing of the index $j$ corresponds to
the positive traverse of the boundary $\partial Q$; we will name
such quadruples $(z_1,z_2,z_3,z_4)$ admissible.

One of the main geometric characteristics of a quadrilateral
$\mbox{\boldmath$Q$}$ is its {\it conformal modulus}. There are
many equivalent definitions of this concept. If $f$ is a conformal
mapping of $Q$ onto a rectangle $[0,1]\times [0,h], h >0,$ such
that the points $z_1$, $z_2$, $z_3$, and $z_4$  correspond to $0$,
$1$, $1+ih$, and $ih$, then $h$ is uniquely defined   and is
called the conformal modulus of $\mbox{\boldmath$Q$}$ \cite[p.52,
Def. 2.1.3]{papa}. In this case, we write
$$
h=\text{Mod}(\mbox{\boldmath$Q$}).
$$
Another way to define the modulus is to use the concept of the
extremal length $\lambda(\Gamma)$ of a curve-family $\Gamma$ (see,
e.g. \cite{ahlfors}). If $\Gamma$ is the family of curves  in the domain $Q$,
connecting the sides $z_1z_2$ and
$z_3z_4$ of  the quadrilateral $\mbox{\boldmath$Q$}$, then
$\text{Mod}(\mbox{\boldmath$Q$})=\lambda(\Gamma)$. If we consider
the family $\Gamma_1$ of curves connecting the sides $z_2z_3$ and
$z_1z_4$ in $Q$, then
$\text{Mod}(\mbox{\boldmath$Q$})=1/\lambda(\Gamma_1)$. At last,
$$
\text{Mod}(\mbox{\boldmath$Q$})=\left(\inf_u\int\!\!\!\int_Q|\nabla
u|^2dx dy\right)^{-1}
$$
where the infimum is taken over all smooth functions \ $u$ \ in
$Q$, continuous on $\overline{Q}$, which are equal to $0$ on the
boundary arc $z_1z_2$ and equal to $1$ on the boundary arc
$z_3z_4$. Moreover, it is known that the infimum here is attained
by a harmonic function \cite[p. 65]{ahlfors}.

Assume that $Q$ is a bounded Jordan domain in $\mathbb{C}$,
$L=\partial Q$,  and $z_1$, $z_2$, $z_3$, $z_4$ are some points on
$L$ satisfying the above requirements. Consider the quadrilateral
$\mbox{\boldmath$Q$}=(Q;z_1,z_2,z_3,z_4)$. Then we will say that
the value of $\text{Mod}(Q;z_1,z_2,z_3,z_4)$ is {\it the interior
modulus}. We can also consider the quadrilateral
$\mbox{\boldmath$Q$}^c:=(Q^c;z_4,z_3,z_2,z_1)$ where $Q^c$ is the
complement of $Q$ with respect to the Riemann sphere. In that case,
we will name $\text{Mod}(\mbox{\boldmath$Q$}^c)$ {\it the exterior
modulus}.

The main topic of this paper is to analyze the interior and
exterior moduli of polygonal quadrilaterals and to apply the
results to geometric function  theory. Because the upper
half-plane can be conformally mapped onto a polygonal domain in
terms of the Schwarz-Christoffel transformation, both the interior
and exterior moduli of polygonal quadrilaterals can be computed.
This transformation is semi-explicit, there are so called
accessory parameters that have to be determined separately for
each case \cite{af,ky}. There are no universal recipes for finding
the accessory parameters, which itself leads to
 numerically ill-conditioned problems. The best methods for numerical computation
 of Schwarz-Christoffel type conformal mappings are those
 developed and implemented  by T.A. Driscoll and L.N. Trefethen \cite{dt}. For a survey of the
 available methods, see N.~Papamichael and N.~Stylianopoulos \cite{papa}.

We study the conformal mapping from upper half plane onto the
exterior of a convex polygonal quadrilateral with vertices
$0,1,a,b$ with ${\rm Im}\,a>0, {\rm Im}\,b>0$
 from analytic point of view and our goal is to explicitly find the
Schwarz-Christoffel mapping and its parameters.  Recall first that by classical
theory of elliptic functions it is known that the upper half plane is conformally mapped
onto a rectangle under the inverse of the elliptic function $sn\,$ and the ratio of
the side lengths is given by a quotient of complete elliptic integrals \cite{af,ky,papa}. This fact
extends to the case
of conformal mapping of the upper half plane onto parallelograms and the mapping is given by
 generalized elliptic functions, now depending on the least angle of the parallelogram, and the
interior modulus is given by a quotient of generalized complete elliptic integrals as shown in
\cite[Section 2]{aqvv}. These generalized complete elliptic integrals were introduced in
\cite[p.158]{bb} and we mention in passing that during the
past decade these integrals have been studied intensively, cf. e.g. \cite{cqw, qmc, qmb}
and the bibliographies therein.
 A further extension of the above  conformal mapping problem is to map
the upper half plane onto a convex polygonal quadrilateral and
such a mapping is given by the Schwarz-Christoffel transformation
expressed in terms of the Gaussian hypergeometric function
${}_2F_1(a,b;c;z)$ with parameters depending on the angles; the
previous case is a special case of this one as shown in
\cite[Section 2]{hvv}. The conformal mapping problem onto the
exterior of a polygonal quadrilateral that we study here is much
more difficult. In one of our main results, Theorem~\ref{modul},
we prove  that the mapping can be expressed in terms of the
Lauricella hypergeometric function $F_D^{(3)}\,.$  A particular
case of this mapping was studied by P.~Duren and J.~Pfaltzgraff
\cite{dp}, namely a conformal mapping of the upper half plane onto
the exterior of a rectangle.  As pointed out in \cite{hvv}, this
mapping already appeared in the works of W.~Burnside \cite{bu}.  We
also give a Mathematica function for the computation of the
exterior modulus based on Theorem~\ref{modul} and compare its
numerical precision by a comparison to the recent numerical computation results
reported in \cite{nasser} and observe very good agreement of the results.

Conformal moduli play an important role in geometric function
theory and applications, in particular, they are valuable tools in
the study of quasiconformal mappings (see, e.g. \cite{ahlfors,avv,
d, gh, kuhnau}). One of the definitions of quasiconformal mappings
(the so-called geometric definition) uses the moduli. A
sense-preserving homeomorphism of the Riemann sphere
$\overline{\mathbb{C}}$ is called $K$-quasiconformal ($K\ge 1$) if
it satisfies the following condition: conformal moduli are
quasiinvariant under the mapping, i.e. if
$\mbox{\boldmath$Q$}=(Q;z_1,z_2,z_3,z_4)$ is an arbitrary
quadrilateral and
$f(\mbox{\boldmath$Q$})=(f(Q);f(z_1),f(z_2),f(z_3),f(z_4))$, then
\begin{equation}\label{1}
K^{-1}\text{Mod}(\mbox{\boldmath$Q$})\le
\text{Mod}(f(\mbox{\boldmath$Q$}))\le
K\text{Mod}(\mbox{\boldmath$Q$}).
\end{equation}

A closed Jordan curve $L$ on the Riemann sphere is called {\it a
quasicircle} if it is the image of the unit circle under a
quasiconformal mapping defined in the whole plane.
If we know $K$ such that \eqref{1} is valid,  then $L$ is called a $K$-quasicircle.
An important problem is to determine, for a given Jordan  curve, whether it is a quasicircle
or not, and, if the answer is positive, to either find or estimate the minimal
possible value of $K$,  denoted by $K_L$. The problem of finding  $K_L$ is open also for the case of
curves $L$ as  simple as the boundaries of long rectangles and we will discuss this below \cite[p. 455 Probl. (20)]{hkv}.

Ahlfors \cite{ahlfors:qc2} gave the following geometric
characterization for quasicircles. If $L$ is a closed Jordan curve
in the plane and there is a constant $C$ such that for every three
points $z_1$, $z_2$, and $z_3$ on $L$ such that $z_3$ lies on the
subarc of $L$  with smaller diameter and with endpoints $z_1$ and
$z_2$, the inequality
\begin{equation}\label{CK}
|z_{1}-z_{3}|+|z_{2}-z_{3}|\leq C|z_{1}-z_{2}|
\end{equation}
holds, then $L$ is a $K$-quasicircle where $K$ depends only on $C$
\cite[p.23, Def. 2.2.2, Thm 2.2.5]{gh}. Conversely, if $L$ is a
$K$-quasicircle, then for every appropriate  triple on $L$
the inequality \eqref{CK} holds with $C$ depending only on $K$.
The monograph of Gehring and Hag \cite{gh} gives many more characterizations
of quasicircles and provides a survey of their many applications.

If $L$ is a quasicircle, then there is a quasiconformal reflection
with respect to $L$, i.e. a sense-reversing quasiconformal
automorphism $g$ of the  Riemann sphere which keeps every point of
$L$ fixed and maps the bounded complementary component of $L$ onto
the unbounded one and vice versa.  Another important problem is to
find or estimate the minimal coefficient of quasiconformality for
such a mapping $g$; further we will denote this coefficient of
quasiconformal reflection by $QR_L$.  This is a very difficult
problem even for polygonal curves in $\mathbb{C}$ studied by R.
K\"uhnau in a series of papers, e.g. \cite{kuhnau1}.  A nice survey
is given in \cite[pp.525-531]{kruskal}. Here we give some of these
results. Since every such a curve $L$ in $\mathbb{C}$ determines
its interior in a unique way, we will also say that $QR_L$  is the
coefficient for $Q:=\mbox{\rm int}(L)$.  If there exists a circle
tangent to every side of a closed polygon $L$, then
$QR_L=2/\alpha-1$ where $\pi\alpha$ is the least interior angle of
$L$. In particular, for triangles the problem is solved. For
quadrilaterals, the problem is open even for the case of rectangles.
The value of $QR_L$ is known only for rectangles $[0,a]\times
[0,1]$ close to a square \cite{werner}: if $1\le a < 1.037$, then
$QR_L=3$. For sufficiently long rectangles with $a>2.76$ it is
proven \cite{werner} that $QR_L>3$. Moreover, for any $a>1$ the
estimate
\begin{equation}\label{piest}
\frac{\pi}{4}\,a<QR_L<\pi a
\end{equation}
holds \cite{werner}, see also \cite{kuhnau,kruskal}.

The value $QR_L$ is closely connected with
$$
M_L:=\sup
\frac{\text{Mod}(\mbox{\boldmath$Q$})}{\text{Mod}(\mbox{\boldmath$Q$}^c)}
$$
where $\mbox{\boldmath$Q$}=(Q;z_1,z_2,z_3,z_4)$, $\partial Q=L$
and the supremum is taken over all admissible quadruples
$(z_1,z_2,z_3,z_4)$ on $L$. As K\"uhnau noted in \cite{kuhnau2},
\begin{equation}\label{mqr}
QR_L\ge M_L,
\end{equation}
therefore, every lower estimation for  $M_L$ also gives a lower
estimation for  $QR_L$.  For related resuts,  see Shen  \cite{sh}.

In the present paper, we investigate the problem of estimations of
$M_L$ and $QR_L$ for isosceles trapezoidal curves. With the help
of the Schwarz-Christoffel formula, we construct conformal
mappings of the upper half-plane onto the interior and exterior of
$L$. Comparison of the interior and exterior moduli for some
quadruples of points $z_1$, $z_2$, $z_3$, $z_4\in L$ allows us to
obtain lower estimates for $M_L$, and, therefore, for $QR_L$.

In
addition, using fairly simple methods, we get two-sided estimates
of $QR_L$ and $M_L$ for isosceles trapezoidal polygons $L$ of height $1$ in
terms of lengths of its sides and angles. In particular, our main result about quasiconformal reflection,
Theorem~\ref{cd0}, states that for such $L$ with acute angle
$\pi\alpha$ and bases $c$ and $d$, $d=c+\cot(\pi\alpha)$, the following estimates hold:
$$
M_L\ge \left\{
\begin{array}{rl}
g(\lambda_0)(1+C(\alpha))d,\ \  &\text{if} \ \ \frac{c}{d}\ge
\lambda_0,\\[2mm]
g(\lambda)(1+C(\alpha))d,  \ \ &\text{if} \ \
\frac{c}{d}<\lambda_0.
\end{array} \right.
$$
Here $g(\lambda)=\lambda \K\left(\sqrt{1-\lambda^2}\right)/\K(\lambda)$, $\K(\lambda)$ is the complete elliptic integral of the first kind, $\lambda_0=0.7373921\ldots$ is the unique point of maximum of $g$ on $(0,1),$ $g(\lambda_0)=0.708434...$ and
$$
C(\alpha)=\left(\sqrt{1+\tan^2(\pi\alpha)/4}-\tan(\pi\alpha)/2\right)^2,\quad
0< \alpha\le 1/2.
$$

At last, in Section \ref{sec:ss}, with the help of the concept of
strongly starlike curve, we give upper bounds for $QR_L$ for
isosceles trapezoidal polygons.

\section{Hypergeometric functions and their generalizations}

To find some explicit formulas for  the interior and exterior
conformal moduli for a trapezoidal curve we need special
functions. In this section we recall the Gaussian and  Appell
hypergeometric functions and some of their generalizations
\cite{bateman,bf,avv}.

First we recall that, for $|z|<1$,
the Gaussian hypergeometric function is defined by the equality
$$
{}_2F_1(a,b;c;z)=\sum_{n=1}^\infty
\frac{(a)_n(b)_n}{(c)_n}\,\frac{z^n}{n!}\,,
$$
for $|z|<1$, where
$(q)_n$ denotes the Pochhammer symbol, i.e.
$(q)_n=q(q+1)\ldots(q+(n-1))$ for every natural $n$ and $(q)_0=1$.
It can be extended analytically to the domain $|z|>1$ along any
path avoiding $1$ and $\infty$.  Moreover, we have
\begin{equation}\label{hyperint}
\int_0^1t^{a-1}(1-t)^{c-a-1}(1-tz)^{-b}dt=B(a,c-a)
\,{}_2F_1(a,b;c;z)
\end{equation}
where $B(\,\cdot\,,\,\cdot\,)$ is the Euler beta function and the
integral in \eqref{hyperint} converges if $\Re c>\Re a>0$. We also
recall that the beta function can be expressed via the Euler
gamma function:
$$
B(\alpha,\beta)=\,\frac{\Gamma(\alpha)\Gamma(\beta)}{\Gamma(\alpha+\beta)}\,.
$$

The complete elliptic integrals $
{\K}, {\E}$ of the first and second kinds
\begin{equation}\label{ek}
{\K}
(\lambda)=\int_0^1\frac{dt}{\sqrt{(1-t^2)(1-\lambda^2t^2)}}\,,
\quad {\E} (\lambda)=\int_0^1{\sqrt{\frac{1-\lambda^2t^2}{1-t^2}}}
\,{dt}
\end{equation}
are, in fact, special cases of the Gaussian hypergeometric
function; we have
$$
{\K}
(\lambda)=\frac{\pi}{2}\,\,{}_2F_1({\textstyle\frac{1}{2}},{\textstyle\frac{1}{2}};1;\lambda^2),\quad {\E}
(\lambda)=\frac{\pi}{2}\,\,{}_2F_1(-{\textstyle\frac{1}{2}},{\textstyle\frac{1}{2}};1;\lambda^2).
$$
For the decreasing homeomorphism $\mu:(0,1)\to (0,\infty)$
$$
\mu(r)=\frac{\pi}{2} \frac{{\K}(\sqrt{1-r^2})}{{\K}(r)}\,,\quad
0<r<1,
$$
the following differentiation formulas hold \cite[p.475]{avv}
\begin{equation}\label{mudiff}
\mu'(r)=-\frac{\pi^2}{4 r(1-r^2) {\K}^2(r)},\quad
\mu''(r)=\frac{\pi^2 (  2{\E}(r) -(1+r^2){\K}(r))}{4 r^2(1-r^2)^2
{\K}^3(r)}\,.
\end{equation}
Let ${\mathbb{H}}^2 = \{z :  {\rm Im} z >0\}$ be the upper
half-plane and $0<r<1.$ Then the modulus of the quadrilateral
$\mbox{\boldmath{$H$}}^2_r= ({\mathbb{H}}^2; 0,1,1/r^2,\infty)$
can be written as follows
\begin{equation}\label{Hmod}
\text{Mod}(\mbox{\boldmath{$H$}}^2_r)=
\frac{\K(\sqrt{1-r^2})}{\K(r)}\,.
\end{equation}
This formula follows from the definition, if we use of a canonical
conformal mapping of the the upper half plane onto a rectangle. It
also follows from \cite[7.12, 7.33]{hkv}.

The Appell hypergeometric function $F_1(a,b_1,b_2;c;z,w)$ is
defined as
\begin{equation}\label{appell}
F_1(a;b_1,b_2;c;z,w)=\sum_{m,n=0}^\infty
\frac{(a)_{m+n}(b_1)_m(b_2)_n}{(c)_{m+n}}\,\frac{z^mw^n}{m!\,n!}\,,
\end{equation}
see, e.g. \cite{bateman}.
The series converges in the bidisk $B^2:=\{|z|,|w|<1\}$ and
similar to the case of hypergeometric function, can be continued
analytically outside the bidisk $B^2$ along any path not
containing the points with $z=1$ and $w=1$. The following formula
is due to Picard:
\begin{equation}\label{hyperintapp}
\int_0^1t^{a-1}(1-t)^{\mu-1}(1-tz)^{-b_1}(1-tw)^{-b_2}dt=B(a,\mu)F_1(a;b_1,b_2;a+\mu;z,w).
\end{equation}
The integral converges for $\Re a>0$, $\Re \mu>0$, if $z$, $w\neq
1$. It is evident that if either $b_2=0$ or $w=0$, then
$F_1(a;b_1,b_2;c;z,w)={}_2F_1(a,b_1;c;z)$.

The Lauricella hypergeometric function (\cite{lau}, see also
\cite{sch}) generalizes the Appell hypergeometric function for the
case of arbitrary number $n$ of variables $z_1,\ldots, z_n$:
\begin{equation}\label{fd}
F^{(n)}_D(a;b_1,\ldots,b_n; c; z_1,\ldots, z_n)=
\sum_{m_1,\ldots,m_n\ge0} \frac{(a)_{m_1+\ldots+m_n}
(b_1)_{m_1}\ldots (b_n)_{m_n}}{(c)_{m_1+\ldots+m_n} m_1!\ldots
m_n!}\, z^{m_1}_1\ldots z^{m_n}_n,
\end{equation}
$|z_1|,\ldots, |z_n|<1$. It has the integral representation
\begin{multline}\label{fd1}
F^{(n)}_D(a;b_1,\ldots,b_n; c; z_1,\ldots,
z_n)\\=(B(a,c-a))^{-1}\int_0^1u^{a-1}(1 - u)^{c-a-1}(1 -
uz_1)^{-b_1}\ldots (1 - uz_n)^{-b_n} du,
\end{multline}
$\Re c>\Re a>0$,  which gives an analytic continuation of the
Lauricella hypergeometric function outside  the polydisk
$|z_k|<1,1\le k\le n$, except for the hyperplanes $z_k=1$.

\section{Conformal mappings of the  half-plane onto the exterior of a polygonal quadrilateral }\label{conf1}

Heikkala, Vamanamurthy, and Vuorinen \cite[pp.78-82]{hvv} studied the  Schwarz-Christoffel mapping of the upper half-plane
onto a convex polygonal quadrilateral. In this section we consider the conformal mapping of   the upper half-plane
onto the exterior of a convex polygonal  quadrilateral.

To specify the geometry, we suppose that the bounded
polygonal quadrilateral has
four segments as its sides which form the exterior angles $\pi(1+\alpha)$,
$\pi(1+\beta)$, $\pi(1+\gamma)$, and $\pi(1+\delta)$ at the
vertices, with the angle parameters satisfying the constraints
\begin{equation}\label{abcd}
\alpha+\beta+\gamma+\delta=2,\quad 0<\alpha,\beta,\gamma,\delta<1,
\end{equation}
The conformal mapping $f$ of the upper half plane
onto the complement of this quadrilateral is given
by the generalized Schwarz-Christoffel formula \cite[ Section  5.6, formula  (5.6.3b)]{af}
\begin{equation}\label{F}
f(z)=C_1\int_0^z
\frac{\zeta^\alpha(\zeta-1)^\beta(\zeta-t)^\gamma}{(\zeta-z_0)^2(\zeta-\overline{z}_0)^2}\,d\zeta+C_2 .
\end{equation}
The preimages of the vertices are the points $0$, $1$, $t$
($t>1$), and $\infty$; the point $z_0$ is an unknown pole.
We note that from \eqref{abcd} it follows, in particular, that
the bounded
complementary  polygonal domain is convex, since  the interior angles
at the vertices are  $(1-\alpha)\pi$,
$(1-\beta)\pi$, $(1-\gamma)\pi$, and $(1-\delta)\pi$ and $\alpha$
$\beta$, $\gamma$, $\delta>0$.

To find $z_0=x_0+iy_0$ we derive an equation which is obtained from
the fact that the residue of the integrand at the point $z_0$
vanishes. The integrand is equal to
$$
\frac{g(\zeta)}{(\zeta-z_0)^2}\,, \quad \text{where}\quad
g(\zeta)=\frac{\zeta^\alpha(\zeta-1)^\beta(\zeta-t)^\gamma}{(\zeta-\overline{z}_0)^2}\,,
$$
and hence it has a pole of the second order at the point $z_0.$
Computation yields
$$
\text{res}_{\,\zeta=z_0}\frac{g(\zeta)}{(\zeta-z_0)^2}\,=g'(z_0).
$$
Consequently, $z_0$ must satisfy the equality $g'(z_0)=0$ which is
equivalent to $(\log g(z_0))'=0$. At last,
$$
(\log
g(\zeta))'=\frac{\alpha}{\zeta}+\frac{\beta}{\zeta-1}+\frac{\gamma}{\zeta-t}-\frac{2}{\zeta-\overline{z}_0}\,,
$$
and from this we obtain the desired equation for $z_0$:
$$
\frac{\alpha}{z_0}+\frac{\beta}{z_0-1}+\frac{\gamma}{z_0-t}=\frac{1}{iy_0}\,,\quad
y_0=\Im z_0.
$$

\begin{lem}\label{iy}

In the upper half-plane $\{z:\Im z>0\}$ the equation
\begin{equation}\label{1a}
\frac{\alpha}{z}+\frac{\beta}{z-1}+\frac{\gamma}{z-t}=\frac{1}{iy},
\end{equation}
$z=x+iy$, has a unique solution.
\end{lem}

\begin{proof} We can represent \eqref{1a} in the form
$$
z-\overline{z}=\frac{2z(z-1)(z-t)}{Q(z)}
$$
where
$$
Q(z)=\alpha(z-1)(z-t)+\beta z(z-t)+\gamma
z(z-1)=(\alpha+\beta+\gamma)z^2-(\alpha(1+t)+\beta
t+\gamma)z+\alpha t.
$$
Therefore, $\overline{z}=R(z)$ where
\begin{equation}\label{Rx}
  R(z)=z-\frac{2z(z-1)(z-t)}{Q(z)}\,=\frac{(\alpha+\beta+\gamma-2)z^3-((\alpha-2)(1+t)+\beta
t+\gamma)z^2+(\alpha-2)tz}{Q(z)}.
\end{equation}
Since $R(z)$  is a rational function  with real coefficients, we
have $\overline{R(z)}=R(\overline{z})$ and, therefore, all
solutions of \eqref{1a} are also solutions of the equation
\begin{equation}\label{RR}
z=R(R(z))=R(z)-\frac{2R(z)(R(z)-1)(R(z)-t)}{Q(R(z))}\,.
\end{equation}
The function $R(R(z))$ is a rational function of degree $9$;  for
real $x$ it has poles at the points where $Q(x)=0$ and
$Q(R(x))=0$. For the quadratic function $Q(x)$ we have $Q(0)$,
$Q(t)>0$ and $Q(1)<1$, therefore, it has two real zeroes
$x_1\in(0,1)$ and $x_2\in(1,t)$. Moreover,
$$
\lim_{x\to x_k-}R(x)=-\infty,\quad \lim_{x\to
x_k+}R(x)=+\infty,\quad k=1,2,
$$
and, because of the inequality $\alpha+\beta+\gamma<2$, we have
$$
\lim_{x\to -\infty}R(x)=+\infty, \quad \lim_{x\to
+\infty}R(x)=-\infty.
$$
This implies that on each of the intervals $(-\infty,x_1)$,
$(x_1,x_2)$ and $(x_2,+\infty)$ there are exactly two points where
either  $R(x)=x_1$ or $R(x)=x_2$. Therefore, in addition to $x_1$
and $x_2$, we have six points on the real axis where $R$ has
poles. It is evident that all these eight points, which are poles of
$R(R(x))$, are different. Denote them by $\tau_j$, $1\le j\le 8$;
the points are labelled so  that $\tau_1<\tau_2<\ldots<\tau_8$.

Now we will show that $R'(x)<0$ for real $x$ different from the
points $x_1$ and $x_2$. We have
$$
R(x)=x-\frac{2}{f(x)}, \quad
f(x)=\frac{\alpha}{x}+\frac{\beta}{x-1}+\frac{\gamma}{x-t}\,.
$$
Consequently,
$$
R'(x)=1+\frac{2f'(x)}{f^2(x)}, \quad
f'(x)=-\frac{\alpha}{x^2}-\frac{\beta}{(x-1)^2}-\frac{\gamma}{(x-t)^2}\,<0.
$$
From the Cauchy-Schwarz inequality we obtain
$$
f^2(x)=\left(\frac{\alpha}{x}+\frac{\beta}{x-1}+\frac{\gamma}{x-t}\,\right)^2\le
(\alpha+\beta+\gamma)\left(\frac{\alpha}{x^2}+\frac{\beta}{(x-1)^2}+\frac{\gamma}{(x-t)^2}\right)<2|f'(x)|,
$$
and hence, $R'(x)<0$. It is evident that on  every one of the
intervals $I_j:=(\tau_j,\tau_{j+1})$,  $1\le j\le 7$, $R(R(x))$ is
strictly increasing. Then
$$
\lim_{x\to \tau_j+}(x-R(R(x)))=+\infty, \quad \lim_{x\to
\tau_{j+1}-}(x-R(R(x)))=-\infty.
$$
Therefore,  $x-R(R(x))$ takes all real values on $I_j$ and there
exists a point $s_j\in I_j$ such that $s_j=R(R(s_j))$. From this
we conclude that every $s_j$,  $1\le j\le 7$, satisfies the
equality \eqref{RR}. From \eqref{Rx} we conclude that $R(x)$ is a
rational function of degree $3$ which is the ratio of two
polynomials of degrees $3$ and $2$. Moreover, $R(x)\sim
(1-2(\alpha+\beta+\gamma)^{-1})x$, $x\to\infty$. Therefore,
$R(R(x))$ is a rational function of degree $9$ and
$R(R(x))\sim(1-2(\alpha+\beta+\gamma)^{-1})^2x$, $x\to\infty$.
This implies that the rational function $x-R(R(x))$ has degree
$9$. Since at least $7$ of its zeros are real, we conclude that it
has at most one complex (i.e. non-real) root in the upper
half-plane. Since all solutions of \eqref{1a} satisfy \eqref{RR},
we see that \eqref{1a} has at most one complex root in the upper
half-plane.

To prove that such a root exists, we note that if we fix one side of
the boundary polygon $A_1A_2A_3A_4$, say $A_1A_2,$ and assume that
it coincides with the segment $[0,1]$ of the real axis and change
the length of the  other side, $A_2A_3$ from $0$ to $\infty$ (with
fixed values of angles $\alpha$, $\beta$, and $\gamma$), then the
exterior conformal moduli of the obtained polygons continuously
increase  from $0$ to $+\infty$ and, therefore, there exists such a
polygon $P$ that its modulus coincides with the  modulus of the
upper half-plane with vertices at the points $0$, $1$, $t$, and
$\infty$. Then there exists a conformal mapping $f$ of the upper
half-plane onto the exterior of $P$ such that the points $0$, $1$,
$t$, and $\infty$ are mapped to the  vertices of $P$. Denote
$z_0=f^{-1}(\infty)$. Then $f$ has the form \eqref{F} and $z_0$
satisfies \eqref{1a}.
\end{proof}

Now we will give a formula for the unique solution of \eqref{1a}.
Denote
$$
E=\alpha+\beta+\gamma-1,
$$
$$
A=2 (E-1)^2,
$$
$$
B=(E-1)[4-3(\alpha+\gamma)+(4-3(\alpha+\beta))t],
$$
$$
C=2-3(\alpha+\gamma)+(\alpha+\gamma)^2+2(3-5\alpha-2 \beta-2
\gamma+2\alpha^2+2\alpha \beta+2 \alpha \gamma+\beta \gamma)t$$
$$+(2-3(\alpha+\beta)+(\alpha+\beta)^2)t^2,
$$
$$
D=(1-\alpha)(\alpha+\gamma-1+(\alpha+\beta-1)t)t,
$$
\begin{equation}\label{rho}
\rho(x)=\frac{\alpha t x } {(1-E)x+E(t+1)-\gamma t-\beta}.
\end{equation}

\begin{lem}\label{iy1}
The unique solution of \eqref{1a}, lying in the upper half-plane, has
the form $z_0=x_0+iy_0$ where $x_0$ is the unique solution of the
cubic equation
\begin{equation}\label{4a}
Ax^3+Bx^2+Cx+D=0,
\end{equation}
satisfying the inequality $x^2<\rho(x)$ and
$y_0=\sqrt{\rho(x_0)-x_0^2}$.
\end{lem}

\begin{proof}
Multiplying  both  sides of \eqref{1a} by $z-t$ we obtain
$$
\frac{\alpha(z-t)}{z}+\frac{\beta(z-t)}{z-1}+\gamma=1+\frac{x-t}{iy}\,,
$$
consequently,
$$
\Re\left[\frac{\alpha
t}{z}+\frac{\beta(t-1)}{z-1}\right]=E
$$
and
$$
\frac{\alpha t x}{|z|^2}+\frac{\beta(t-1)(x-1)}{|z-1|^2}=E.
$$
If we denote $\rho=x^2+y^2$, we have
$$
\frac{\alpha t x}{\rho}+\frac{\beta(t-1)(x-1)}{\rho+1-2x}=E
$$
and
$$
\alpha t x (\rho+1-2x)+\rho\beta(t-1)(x-1)=E\rho(\rho+1-2x).
$$

Similarly, multiplying both sides of \eqref{1a} by $z-1$ we obtain
$$
\frac{\alpha(z-1)}{z}+{\beta}+\frac{\gamma(z-1)}{z-t}=1+\frac{x-1}{iy}.
$$
$$
\Re\left[\frac{\alpha}{z}+\frac{\gamma(1-t)}{z-t}\right]=E,
$$
$$
\frac{\alpha x}{\rho}+\frac{\gamma(1-t)(x-t)}{\rho+t^2-2tx}=E,
$$
$$
\alpha x(\rho+t^2-2tx)+\rho\gamma(1-t)(x-t)=E\rho(\rho+t^2-2tx).
$$

Thus, we have a system of two equations with respect to $x$ and
$\rho$:
\begin{equation}\label{2a}
\alpha t x (\rho+1-2x)+\rho\beta(t-1)(x-1)=E\rho(\rho+1-2x).
\end{equation}
\begin{equation}\label{3a}
\alpha x(\rho+t^2-2tx)+\rho\gamma(1-t)(x-t)=E\rho(\rho+t^2-2tx).
\end{equation}
Subtracting \eqref{2a} from \eqref{3a} we obtain
$$
E\rho(t^2-1-2(t-1)x)=\alpha x(1-t)\rho+\alpha
x(t^2-t)+\rho(1-t)((\gamma+\beta)x-\gamma t-\beta)
$$
This is a linear equation with respect to $\rho$. Solving it, we
obtain
\eqref{rho}.
If we substitute this expression into any of the equations
\eqref{2a}, \eqref{3a}, we find a cubic equation \eqref{4a} for
$x_0$.

By Lemma~\ref{iy}, there exists only one root $x_0$ of \eqref{4a}
satisfying the inequality $x^2<\rho(x)$. Then we find
$y_0=\sqrt{\rho^2(x_0)-x_0^2}$.
\end{proof}

\medskip

After finding the value of $z_0$ we can simply express the
coordinates of the vertices in terms of $\alpha$, $\beta$,
$\gamma$, and $t$. Further, we will assume that the vertices of
the boundary polygon, corresponding to the points $0$, $1$, $t$
and $\infty$, are located at the points $A_1=1$, $A_2=0$, $A_3$
and $A_4$. Then the conformal mapping of the upper half-plane onto
the exterior of the polygonal region is given by the formula
\begin{equation}\label{fz}
f(z)= 1-\frac{h(z)}{h(1)},
\end{equation}
where
\begin{equation}\label{hz}
h(z)=\int_0^z \frac{x^\alpha(1-x)^\beta(1-x/t)^\gamma dx}{(1-x
/z_0)^2(1-x /\overline{z}_0)^2}\,.
\end{equation}
(The branch of the integrand is fixed such that it takes positive
values on $(0,1)$). Here $z_0=z_0(t)$ is described in
Lemma~\ref{iy1}. With the help of the Lauricella function
$F^{(n)}_D$, we can write
\begin{equation}\label{h1}
h(1)=
B(1+\alpha,1+\beta)F_D^{(3)}(1+\alpha;-\gamma,2,2;2+\alpha+\beta;1/t,1/z_0,1/\overline{z}_0),
\end{equation}
$$
h(z)=\frac{z^{1+\alpha}}{1+\alpha}\,
F_D^{(4)}(1+\alpha;-\beta,-\gamma,2,2;2+\alpha;z,z/t,z/z_0,z/\overline{z}_0).
$$
The length of the side $A_2A_3$ is given by the formula
\begin{equation}\label{l2}
l_2:=|A_2A_3|=\frac{I}{h(1)}\,, \quad I=\int_1^t \frac{x^\alpha(x-1)^\beta(1-x /t)^\gamma
dx}{(1-x/z_0)^2(1-x/\overline{z}_0)^2}\,.
\end{equation}

After the change of variables $x=1+(t-1)\tau$ we find
\begin{multline}\label{intend}
I=\int_1^t \frac{x^\alpha(x-1)^\beta(1-x /t)^\gamma
dx}{(1-x/z_0)^2(1-x/\overline{z}_0)^2}\,=\frac{(t-1)^{1+\beta+\gamma}|z_0|^4}{t^\gamma|z_0-1|^4}\,
\int_0^1\frac{\tau^\beta(1-\tau)^\gamma(1+(t-1)\tau)^{\alpha}d\tau}{\left(1-\frac{t-1}{z_0-1}\tau\right)^2\left(1-\frac{t-1}{\overline{z}_0-1}\tau\right)^2}\\
=\frac{(t-1)^{1+\beta+\gamma}|z_0|^4}{t^\gamma|z_0-1|^4}\,B(1+\beta,1+\gamma)F^{(3)}_D(1+\beta;-\alpha,2,2;2+\beta+\gamma;-(t-1),\textstyle
\frac{t-1}{z_0-1},\frac{t-1}{\overline{z}_0-1}).
\end{multline}

Denote $r=1/\sqrt{t} \in (0,1)$. From the conformal invariance of
the modulus we obtain that the desired exterior conformal  modulus
\text{Mod}({\mbox{\boldmath{$Q$}}}) coincides with the conformal
modulus of the quadrilateral $\mbox{\boldmath{$H$}}^2_r$ which is
the upper half-plane with
vertices $z_1=0$, $z_2=1$, $z_3=t=1/r^2$, and $z_4=\infty$.
Applying \eqref{Hmod}, we have
\begin{equation}\label{Mod}
\text{Mod}({\mbox{\boldmath{$Q$}}})={\K}(r')/{\K}(r), \quad
r'=\sqrt{1-r^2},
\end{equation}
where ${\K}(r)$ is the complete elliptic integral of the
first kind.

We can also find the lengths of sides $l_3=|A_3A_4|$ and
$l_4=|A_4A_1|$ and the vertices  $A_3$ and $A_4$. It is easy to
verify that
\begin{equation}\label{l34}
l_3=\frac{\sin\pi\alpha+l_2\sin[\pi(\alpha+\beta)]}{\sin\pi\delta}\,,\quad
l_4=\frac{\sin[\pi(\beta+\gamma)]+l_2\sin\pi\gamma}{\sin\pi\delta}\,,
\end{equation}
\begin{equation}\label{a34}
A_3=-l_2e^{-\pi\beta i},\quad A_4=A_3-l_3e^{-\pi(\beta+\gamma)
i}=1+l_4e^{\pi\alpha i}.
\end{equation}

Therefore, we have

\begin{thm}\label{vertices}
Let $f(z)= 1-\frac{h(z)}{h(1)}$ where $h$ is given by \eqref{hz}.
Then $f$ maps conformally the upper half-plane onto the exterior
of the polygonal line $A_1A_2A_3A_4$ where $A_1=1$, $A_2=0$, $A_3$
and $A_4$ are given by \eqref{a34} where $l_3$ and $l_4$ are given
by \eqref{l34} and $l_2$; here $l_2$ is defined by \eqref{l2},
$h(1)$ and  $I$ are given by \eqref{h1} and \eqref{intend}, taking
into account \eqref{h1} and \eqref{intend}. The length of sides of
$A_1A_2A_3A_4$  are $l_1=1$, $l_2$, $l_3$, and $l_4$.
\end{thm}

It is evident that, under the above assumptions, the exterior of
the quadrilateral is defined uniquely by the value of the length
$l_2$ (or $l_4$). From Lemma~\ref{iy1} it follows

\begin{thm}\label{incr}For fixed angles $(1+\alpha)\pi$, $(1+\beta)\pi$, $(1+\gamma)\pi$,
and $(1+\delta)\pi$, the exterior conformal modulus is strictly
increasing as a function of $l_2$.
\end{thm}

Actually, if it is not the case, then we can find two different
(exterior) quadrilaterals with vertices $A_1$, $A_2$, $A^1_3$,
$A^1_4$ and $A_1$, $A_2$, $A^2_3$, $A^2_4$ such that $A_1=1$,
$A_2=0$, and their conformal moduli coincide. Then the conformal
mappings of the upper half-plane onto these exterior
quadrilaterals is defined by the formula
\begin{equation*}
f_k(z)=C_{1k}\int_0^z
\frac{\zeta^\alpha(\zeta-1)^\beta(\zeta-t_k)^\gamma}{(\zeta-z_{0k})^2(\zeta-\overline{z}_{0k})^2}\,d\zeta+1,\quad
k=1,2.
\end{equation*}
Since the moduli of the quadrilaterals are uniquely defined by
$t_k$ and the dependence is strictly monotone, we see that
$t_1=t_2$. Then, by Lemma~\ref{iy1}, we obtain that
$z_{01}=z_{02}$. From the normalization $f_k(1)=0$, $k=1$, $2$ it
follows that $C_{11}=C_{12}$. Thus, $f_1\equiv f_2$ and,
therefore, the exterior quadrilaterals coincides, which
contradicts our assumptions.

Now we will describe how to determine the conformal modulus of the
given exterior polygonal quadrilateral with angles
$(1+\alpha)\pi$, $(1+\beta)\pi$, $(1+\gamma)\pi$, and
$(1+\delta)\pi$, satisfying \eqref{abcd}.

\begin{thm}\label{modul}
For a given exterior quadrilateral, the conformal modulus is given
by the formula \eqref{Mod} where $r=1/\sqrt{t}$ and $t$ is a
unique solution to the equation \eqref{l2} where $h(1)$ and $I$
are defined by \eqref{h1} and \eqref{intend}, keeping in mind that
$z_0=z_0(t)=x_0(t)+i\sqrt{r(x_0(t))-x_0^2(t)}$, $x_0(t)$ is a
solution to the equation \eqref{4a} satisfying the inequality
$x^2<\rho(x)$ and $\rho(x)$ is given by \eqref{rho}.
\end{thm}

\begin{rem}\label{noncon}
If we consider the general (non-convex) case, where $\alpha$,
$\beta$, $\gamma$, and $\delta$ satisfy the conditions
\begin{equation*}
\alpha+\beta+\gamma+\delta=2,\quad
-1<\alpha,\beta,\gamma,\delta<1,
\end{equation*}
instead of \eqref{abcd}, then we obtain that, by Lemma~\ref{iy1},
for every $M$ there are at most three different exterior
quadrilaterals of the form $A_1A_2A_3A_4$, $A_1=1$, $A_2=0$, with
angles $(1+\alpha)\pi$, $(1+\beta)\pi$, $(1+\gamma)\pi$, and
$(1+\delta)\pi$, such that their conformal moduli equal $M$.
\end{rem}

In connection with Remark~\ref{noncon}, we can suggest

\begin{conjecture}\label{conj}
For every $M>0$ there is only one exterior quadrilateral with
given angles and vertices $A_1=1$, $A_2=0$, $A_3$, $ A_4$,
conformal modulus of which equals~$M$.
\end{conjecture}

With the help of Theorem~\ref{modul}, we can calculate the
exterior modulus of a sufficiently arbitrary convex polygonal line
$A_1A_2A_3A_4$ with vertices  $A_1=1$, $A_2=0$, $A_3=A$,  and
$A_4=B$. With the help of Wolfram Mathematica package, we created
a function  ExtMod$[A,B,n,wp]$  which  calculated the values of
the exterior modulus and also gives the values of the parameters
$\alpha$, $\beta$, $\gamma$, $\delta$, $t$, and $z_0$; here $n$
and $wp$ are some additional parameters; their meaning will be
explained below. The code is contained in Appendix~A. Now we
briefly describe its structure.

To find the exterior modulus with the help of \eqref{Mod}, we
determine the value of $t=1/r^2$. For this, we use the bisection
method on the segment $[1,T_2]$, $T_2=10^n$, $n>0$; the number of
iterations is $S=[5(n+15)]$ where $[x]$ denotes the integer part
of $x$. The input consists of coordinates of the vertices $A_3$
and $A_4$, the value of $n$ and the parameter $wp$ that specifies
how many digits of precision should be maintained in internal
computations (the Wolfram Mathematica option
"WorkingPrecision").
The program calculates the values of angles $\alpha$, $\beta$,
$\gamma$, and $\delta$, which are denoted for short by $a$, $b$,
$c$, and $d$. Then, on every iteration, we solve the cubic
equation \eqref{4a} with coefficients corresponding to the current
value of $t$ and determine its unique root $x_0$ satisfying
$x^2<\rho(x)$. After this, we find $y_0$, $z_0=x_0+i y_0$ and
calculate the integrals
\begin{equation}\label{J12}
J_1=\int_0^1 \frac{x^\alpha(1-x)^\beta(1-x/t)^\gamma dx}{|1-x/z_0|^4}\,, \quad J_2=\int_1^t \frac{x^\alpha(x-1)^\beta(1-x/t)^\gamma dx}{|1-x/z_0|^4}\,,
\end{equation}
corresponding to the current value of the parameter $t$, and
compare its ratio $J_2/J_1$ with $L$ which is the length of the
side $|A_2A_3|$. (In fact, $J_2/J_1$ coincides with $l_2$ given by
\eqref{2}.)

Now we give a suggestion how to fix the value of $n$ if we can obtain an a priori
estimate from above for the value of the desired exterior modulus $M$. If $M$
is less than $2.3$, we can take $n=2$. For $2.3\le M\le 11.8$ the
value $n=(10 M-6)/7$ is suitable. We note that for $M>11.8$ the
program does not work properly because of degeneration of the elliptic integrals.

\medskip

Now we will give results of some numeric examples.\medskip

\begin{example}\label{ex1}
In \cite[subsect.~5.3]{nasser} exterior polygonal quadrilaterals with the following vertices are considered:\smallskip

(1) $A_1 = 1$,  $A_2 = 0$, $A_3 = -19/25 + i 21/25$, $A_4 = 28/25
+ i69/50$;\smallskip

(2) $A_1 = 1$,  $A_2 = 0$, $A_3 = -3/25 + i21/25$, $A_4 = 42/25 +
4i$.

\noindent With the use of the boundary integral equations method,
approximate values of their exterior moduli $M_1$ and $M_2$ were
found.

We calculated the moduli $M_1$ and $M_2$ by our method and the
results are given in Table~\ref{tab}  (see also Appendix~A). To
estimate the accuracy of our calculations we also find the moduli
of the conjugate quadrilaterals, $M_1^*$ and $M_2^*$.
Theoretically, the values of $M_j$ coincide with $(M_j^*)^{-1}$
but difference in approximate values shows how the obtained values
are distinct from the exact ones. We see that the values of our
calculations coincide with those from \cite{nasser} with accuracy
$10^{-9}$.  The differences between $M_j$ and $(M_j^*)^{-1}$ do
not exceed $5\cdot 10^{-15}$; this gives a reason to hope that in
the values of moduli we received $14$ correct digits after the
decimal dot.

On Fig.~\ref{grid} we give the exterior quadrilateral and the image of a grid under
the mapping (\ref{fz}) for the case (1).
\medskip

\begin{figure}[ht] \centering
\includegraphics[width=5 in]{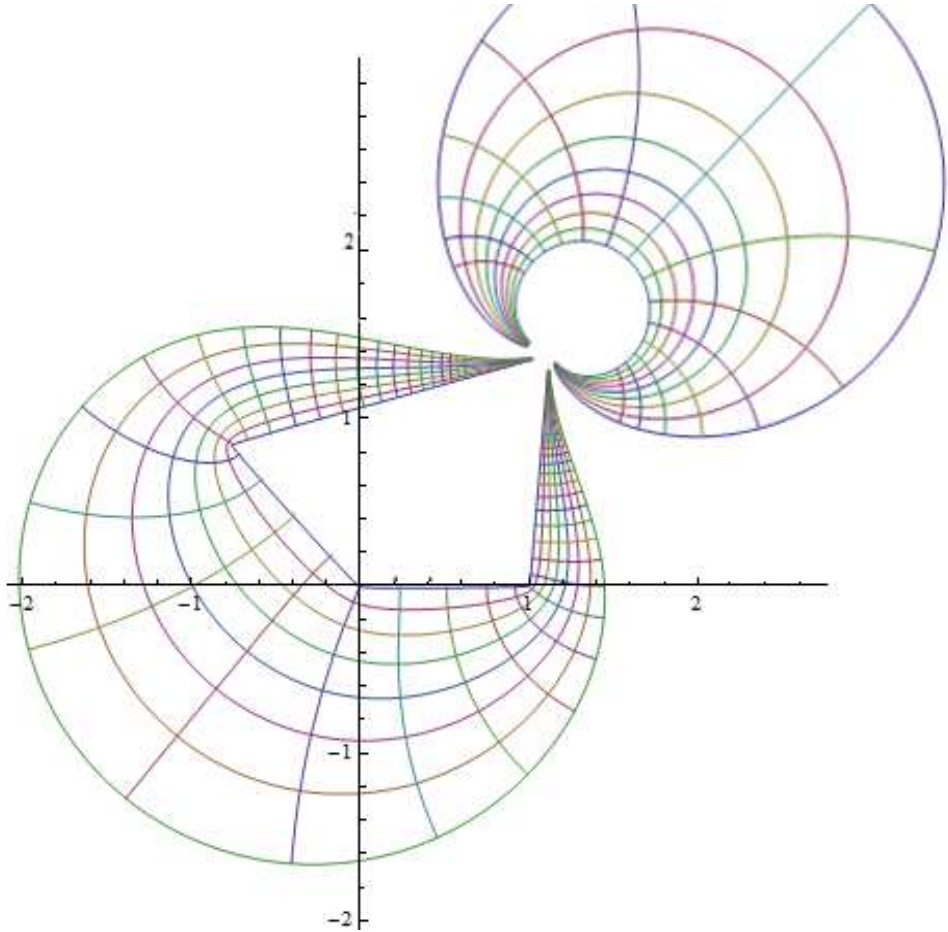}
\caption{The exterior quadrilateral and the image of a grid under
the mapping (\ref{fz}) for Example~\ref{ex1}, Case~(1).}\label{grid}
\end{figure}

\begin{table}[ht]
\caption{The values of exterior moduli for two quadrilaterals.} \centering
\begin{tabular}
{|l|l|l|}
  \hline
& $\qquad \ j=1$ &$\qquad \ j=2$
 \\
  \hline
Approx. values of $M_j$ &  $0.992341633097863$ &$0.959257171919002$
 \\
  \hline
Approx. values of $(M_j^*)^{-1}$ &  $0.992341633097868$&$0.959257171919007$
 \\
  \hline
Approx. values of $M_j$ from \cite{nasser} &  $0.9923416331$ & $0.9592571729$
 \\
  \hline
\end{tabular}\label{tab}
\end{table}

\end{example}

\begin{example}\label{ex2}
Now consider the case of the exterior of a rectangle with vertices
$A_1 = 1$,  $A_2 = 0$, $A_3 = i H$, $A_4 = 1 + iH$, $H>0$.   The
function
\begin{equation}\label{kt}
w=\frac{1}{k}\,\frac{z-\sqrt{t}}{z+\sqrt{t}}\,, \quad k=\frac{\sqrt{t}-1}{\sqrt{t}+1}\,,
\end{equation}
maps the upper half-plane onto itself with the correspondence of points
$0\mapsto-1/k$, $1\mapsto -1$, $t\mapsto 1$, $\infty \mapsto 1/k$. According to the Duren-Pfaltzgraff formula \cite[sect.~(iv)]{dp},
\begin{equation}\label{DP}
H=\frac{2 \E(k)-(1-k)\K(k)}{\E'(k)-k \K'(k)}\,.
\end{equation}
Here $\K(k)$  and $\E(k)$ are the complete elliptic integrals given by \eqref{ek}, $\K'(k)=\K(k')$,   $\E'(k)=\E(k')$, where $k'=\sqrt{1-k^2}$.
Using \eqref{DP}, we can check the accuracy of our calculation. For a given $H$, we find the approximate  value of $t$, then, with the help of \eqref{kt}, determine $k$
and by \eqref{DP} find the value of $H_{\text{app}}$ corresponding to the found approximate values of parameters. Comparing $H_{\text{app}}$ with the initial value of $H$,
we can estimate the accuracy of the approximate method.

In Table~\ref{tab2}, for some $H$ we give the values  of the
exterior moduli $M$ and  the corresponding $H_{\text{app}}$. It is
interesting that the method gives very good results even for very
large $H$. Comparing $H_{\text{app}}$ with $H$ shows that the
accuracy of results for large $H$ is much better than those
obtained by considering the conjugate modulus $M^{*}$ (we can
simply put $H^{-1}$ instead of $H$) and funding, after this, the
reciprocal value $(M^{*})^{-1}$.

\begin{table}[ht]
\caption{The values of exterior moduli $M$ and $k$ for some rectangles.} \centering
\begin{tabular}
{|r|l|l|l|}
   \hline
$H$ & \quad \quad \quad $H_{\text{app}}$ & \quad \quad \quad  $M$ & \quad  \quad \quad $k$ \\
  \hline
1 &  0.999999999999984 & 0.999999999999997& 0.1715728752538083 \\
  \hline
2 & 1.999999999999971 & 1.154924858699707 & 0.2589511664373517  \\
  \hline
3 & 2.999999999999959  & 1.254423186704834 & 0.3183618249446048 \\
  \hline
4 &  3.999999999999940 & 1.328560829309608 & 0.3630445515606185 \\
  \hline
5 & 4.999999999999927  &1.387897041604210 & 0.3985903936736862  \\
  \hline
10 &  9.999999999999874 & 1.580900257847724 &  0.5096661128249422\\
  \hline
50 &  49.99999999999927 &2.062779488244626  & 0.7306010544314864 \\
  \hline
100 &  99.99999999999962 &  2.278195883070594 & 0.7996714751224258 \\
  \hline
$\vphantom{10^{4^3}} 10^3$ & 999.9999999999126  & 3.005361525457626 & 0.9312093496761309 \\
  \hline
$\vphantom{10^{4^3}} 10^4$ & 10000.00000000519  &3.737506317586474  & 0.9776888723313666  \\
  \hline
$\vphantom{10^{4^3}} 10^5$ & 100000.0000021733 & 4.470341757015527  & 0.9928890530750033  \\
 \hline
$\vphantom{10^{4^3}} 10^6$ & 1000000.000038778 & 5.203265238854191  & 0.997745791670292  \\
 \hline
\end{tabular}\label{tab2}
\end{table}

From Table~\ref{tab2} we see that for $H<100$ the $| H_{\text{app}}-H |<10^{-13}$. For large $H$, the  relative error grows but even for $H=10^6$ it less than $10^{-10}$ what can be considered a very good result.
\end{example}

\section{Conformal mappings of the interior and exterior of an
isosceles trapezoidal polygon onto the half-plane}\label{conf}

\begin{mysubsection}
{\bf Interior of a trapezoidal polygon.}\label{intt}
\end{mysubsection}

Further, for convenience, in investigation of the interior and
exterior moduli for considered trapezoidal lines,  we will assume
that one of its parallel sides is on the real axis and, therefore,
$A_1A_2$ does not coincide with the segment $[0,1]$. This
assumption does not play a significant role because it is evident
that the formulas for the corresponding conformal mappings can be
simply obtained from each other by applying conformal
automorphisms of the complex plane of the form $z\mapsto
a_0z+b_0$, $a_0$, $b_0\in \mathbb{C}$. So, let $L$ be the boundary
of a trapezoid with vertices $A_1(-d-i)$, $A_2(-c)$, $A_3(c)$ and
$A_4(d-i)$. Here $d>c>0$ (Fig.~\ref{fig2c}).

Denote by $T^+=T^+(c,\alpha)$ and $T^-=T^-(c,\alpha)$ the interior
and the exterior of $L$. Let $\alpha$ be the value of the angle of
$T^+$ at $A_4$. Then the angles  of $T^+$ at $A_1$, $A_2$, $A_3$,
and $A_4$ are equal $\pi \alpha$, $\pi(1-\alpha)$,
$\pi(1-\alpha)$, and $\pi\alpha$. The corresponding angles of
$T^-$ are equal $\pi(2- \alpha)$, $\pi(1+\alpha)$,
$\pi(1+\alpha)$, and $\pi(2-\alpha)$. Moreover, $d-c=\cot
(\pi\alpha)$ and the both non-horizontal sides are of length equal
to $1/\sin(\pi\alpha)$.

\begin{figure}[ht] \centering
\includegraphics[width=4.5 in]{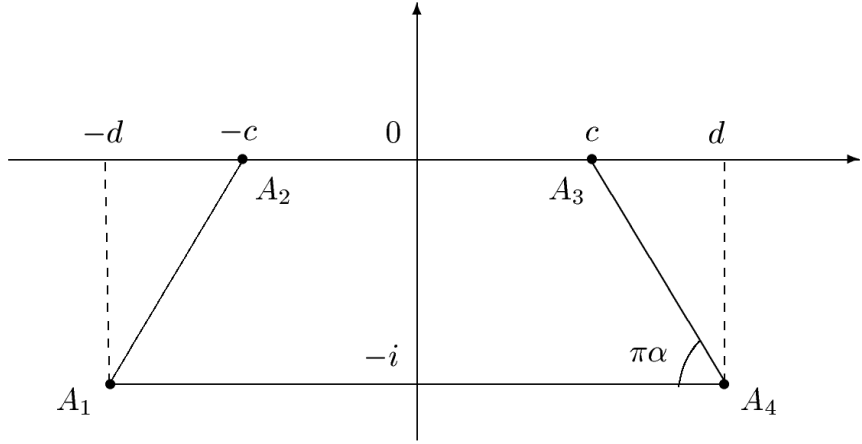}
\caption{Trapezoidal line $L$.}\label{fig2c}
\end{figure}
\medskip

Let us map the lower half-plane  conformally onto $T^+$ such that
$-(1/\lambda)\mapsto A_1$, $-1\mapsto A_2$,  $1\mapsto A_3$, and
$1/\lambda\mapsto A_4$. Here $0<\lambda<1$ is some number
depending on the modulus of $T^+$. According to the
Schwarz-Christoffel formula, the mapping is given by the formula
$$
f^+(z)=C\int_{0}^{z}(t^2-1)^{-\alpha}(\lambda^2t^2-1)^{\alpha-1}\,dt
$$
with the constant
$$
C=c/I,\quad
I=\int_{0}^{1}(1-t^2)^{-\alpha}(1-\lambda^2t^2)^{\alpha-1}\,dt.
$$
Comparing the lengths of the bases, we obtain
\begin{equation}\label{tplus}
\frac{\int_{0}^{1}(1-t^2)^{-\alpha}(1-\lambda^2t^2)^{\alpha-1}\,dt}
{\int_{1/\lambda}^\infty(t^2-1)^{-\alpha}(\lambda^2t^2-1)^{\alpha-1}\,dt}\,=\frac{c}{d}\,.
\end{equation}
From \eqref{tplus} we can find the parameter $\lambda$ and the
conformal modulus of $T^+$:
$$
\text{mod}\,(T^+)=\frac{2{\K}(\lambda)}{{\K}(\lambda')}\,,\quad
\lambda'=\sqrt{1-\lambda^2},
$$
where ${\K}(\lambda)$ is the complete elliptic integral of
the first kind defined in \eqref{ek}.

Now we will express the above integrals  \eqref{tplus} via
the Gaussian hypergeometric function.

The change of variables $\tau=t^2$ and
\eqref{hyperint} yield
\begin{equation}\label{int1}
\int_{0}^{1}(1-t^2)^{-\alpha}(1-\lambda^2t^2)^{\alpha-1}\,dt=
\frac{1}{2}\,B(1-\alpha,{\textstyle\frac{1}{2}})\,
{}_2F_1({\textstyle\frac{1}{2}},1
-\alpha;{\textstyle\frac{3}{2}}-\alpha;\lambda^2).
\end{equation}

Another change of variables $s=1/(\lambda t)$ and the substitution
$\tau=s^2$  lead to
\begin{multline}\label{int2}
\int_{1/\lambda}^\infty(t^2-1)^{-\alpha}(\lambda^2t^2-1)^{\alpha-1}\,dt=\lambda^{2\alpha-1}\int_{0}^{1}(1-s^2)^{\alpha-1}(1-\lambda^2s^2)^{-\alpha}\,ds
\\= \frac{1}{2}\,B(\alpha,{\textstyle\frac{1}{2}})\,\lambda^{2\alpha-1}{}_2F_1({\textstyle\frac{1}{2}},
\alpha;\alpha+{\textstyle\frac{1}{2}};\lambda^2).
\end{multline}

The relations \eqref{int1} and \eqref{int2} allow us to write
\eqref{tplus} in the form
\begin{equation}\label{dpluseq1}
\frac{B(\alpha,{\textstyle\frac{1}{2}})}{B(1-\alpha,{\textstyle\frac{1}{2}})}\,\frac{\lambda^{2\alpha-1}{}_2F_1({\textstyle\frac{1}{2}},
\alpha;
\alpha+{\textstyle\frac{1}{2}};\lambda^2)}{{}_2F_1({\textstyle\frac{1}{2}},1
-\alpha;{\textstyle\frac{3}{2}}-\alpha;\lambda^2)}= \frac{d}{c}\,.
\end{equation}
From the geometric reasoning we conclude that \eqref{dpluseq1} has
a unique solution $\lambda$ on $(0,1)$.

We will also need the boundary correspondence between points of
the real axis and points on the bases of the trapezoid. For $x\in
[0,1]$, with the help of the change of variables, $t=xs$,
$s^2=\tau$ and \eqref{hyperintapp}, we have
\begin{multline*}
\int_{0}^{x}(1-t^2)^{-\alpha}(1-\lambda^2t^2)^{\alpha-1}\,dt=x\int_{0}^{1}(1-x^2s^2)^{-\alpha}(1-\lambda^2x^2s^2)^{\alpha-1}\,ds\\
=xF_1({\textstyle\frac{1}{2}};\alpha,1-\alpha;{\textstyle\frac{3}{2}};x^2,x^2\lambda^2).
\end{multline*}
where $F_1$ is the Appell hypergeometric function \eqref{appell}.
Similarly, for $x>1/\lambda$ we obtain
\begin{multline*}
\int_{x}^{\infty}(t^2-1)^{-\alpha}(\lambda^2t^2-1)^{\alpha-1}\,dt=x\int_{0}^{1}(x^2-s^2)^{-\alpha}(\lambda^2x^2-s^2)^{\alpha-1}\,ds\\
=\lambda^{2(\alpha-1)}x^{-1}F_1({\textstyle\frac{1}{2}};\alpha,1-\alpha;{\textstyle\frac{3}{2}};x^{-2},x^{-2}\lambda^{-2}).
\end{multline*}

\begin{mysubsection}{\bf Exterior of a trapezoidal polygon.}\label{extt}
\end{mysubsection}
Now we describe the conformal mapping of the upper half-plane onto
the exterior $T^-$ of a trapezoidal polygon. We will assume that
the pole of the mapping function is at the point $i$. Let $A_1$,
$A_2$, $A_3$, and $A_4$ correspond to the points $-b$, $-a$, $a$,
and $b$ for some $0<a<b$.

According to the generalized Schwarz-Christoffel formula
\eqref{F}, we have
$$
f^-(z)=\widetilde{C}\int_{0}^{z}\frac{(t^2-a^2)^{\alpha}(t^2-b^2)^{1-\alpha}}{(1+t^2)^2}\,dt
$$
with the constant
$$
\widetilde{C}=c/\widetilde{I}, \quad
\widetilde{I}=\int_{0}^{a}\frac{(t^2-a^2)^{\alpha}(b^2-t^2)^{1-\alpha}}{(1+t^2)^2}\,dt .
$$
Because the residue of the integrand  vanishes at the point
$t=i$, we deduce that the values of $a$ and $b$ are connected by
the equality
\begin{equation}\label{resid}
\frac{\alpha}{1+a^2}\,+\,\frac{1-\alpha}{1+b^2}\,=\frac{1}{2}\,.
\end{equation}
Denote $k=a/b$, $0<k<1$.

Comparing the lengths of the sides we obtain
$$
\frac{\displaystyle\int_{0}^{a}\frac{(a^2-t^2)^{\alpha}(b^2-t^2)^{1-\alpha}}{(1+t^2)^2}\,dt}
{\displaystyle\int_{b}^{\infty}\frac{(t^2-a^2)^{\alpha}(t^2-b^2)^{1-\alpha}}{(1+t^2)^2}\,dt}=\frac{c}{d}\,.
$$
After the change of variables $t=as$, we have
\begin{equation}\label{tminus}
\displaystyle\frac{\displaystyle\int_{0}^{1}\frac{(1-s^2)^{\alpha}(1-k^2s^2)^{1-\alpha}}{(1+a^2s^2)^2}\,ds}
{\displaystyle\int_{1/k}^\infty\frac{(s^2-1)^{\alpha}(k^2s^2-1)^{1-\alpha}}{(1+a^2s^2)^2}\,ds}=\frac{c}{d}\,.
\end{equation}
From \eqref{resid}
we find that, for a fixed
$\alpha$,
\begin{equation}\label{ak}
a^2=a^2(k)=\sqrt{A^2+k^2}-A,\quad
A=\textstyle(\frac{1}{2}-\alpha)(1-k^2).
\end{equation}
Solving \eqref{tminus}, we find the value of $k$ and then
$$
\text{mod}\,(T^-)=\frac{2{\K}(k)}{{\K}(k')}\,, \quad
k'=\sqrt{1-k^2}.
$$

Now we will write the integrals from \eqref{tminus} through
special functions. After  the change of variable $s^2=\tau$,
with the help of \eqref{hyperintapp}, we obtain:
\begin{equation*}
\int_{0}^{1}\frac{(1-s^2)^{\alpha}(1-k^2s^2)^{1-\alpha}}{(1+a^2s^2)^2}\,dt=\frac{1}{2}\,
B({\textstyle\frac{1}{2}}, 1 +
\alpha)F_1({\textstyle\frac{1}{2}};\alpha-1,2;
{\textstyle\frac{3}{2}} + \alpha; k^2,-a^2).
\end{equation*}
Similarly,
\begin{multline*}
\displaystyle\int_{1/k}^\infty\frac{(s^2-1)^{\alpha}(k^2s^2-1)^{1-\alpha}}{(1+a^2s^2)^2}\,ds=k^{3-2\alpha}\int_0^1\frac{(1-k^2\tau^2)^\alpha(1-\tau^2)^{1-\alpha}}{(a^2+k^2\tau^2)^2}\\
=\frac{k^{3-2\alpha}}{2a^4}\,B({\textstyle\frac{1}{2}},2-\alpha)F_1({\textstyle\frac{1}{2}};-\alpha,2;{\textstyle\frac{5}{2}}-\alpha;k^2,-a^2k^2).
\end{multline*}

Therefore, we have the equation to determine $k$:
\begin{equation}\label{k}
\frac{k^{3-2\alpha}B({\textstyle\frac{1}{2}},2-\alpha)}{2a^4B({\textstyle\frac{1}{2}},
1 +
\alpha)}\frac{F_1({\textstyle\frac{1}{2}};-\alpha,2;{\textstyle\frac{5}{2}}-\alpha;k^2,-a^2k^2)}{F_1({\textstyle\frac{1}{2}};\alpha-1,2;
{\textstyle\frac{3}{2}} + \alpha; k^2,-a^2)}=\,\frac{d}{c}
\end{equation}
where $a=a(k)$ (see \eqref{ak}).

Now, as in the case of the interior modulus, we find the relations
between boundary points of the half-plane and points of the sides
of~$T$. We have for $x\in(0,1)$:
\begin{multline*}
\int_{0}^{x}\frac{(1-s^2)^{\alpha}(1-k^2s^2)^{1-\alpha}}{(1+a^2s^2)^2}\,ds=
x\int_{0}^{1}\frac{(1-x^2s^2)^{\alpha}(1-k^2x^2s^2)^{1-\alpha}}{(1+a^2x^2s^2)^2}\,ds\\=
xF^{(3)}_D(1;-\alpha,\alpha-1,2;2; x^2,k^2x^2,a^2x^2)
\end{multline*}
where $F^{(3)}_D$ is the Lauricella hypergeometric function
(see \eqref{fd} and \eqref{fd1}).

For $x>1/k$ we have
\begin{multline*}
\int_{x}^{\infty}\frac{(s^2-1)^{\alpha}(k^2s^2-1)^{1-\alpha}}{(1+a^2s^2)^2}\,ds=
x\int_{0}^{1}\frac{(1-x^2s^2)^{\alpha}(1-k^2x^2s^2)^{1-\alpha}}{(1+a^2x^2s^2)^2}\,ds\\
=\frac{k^{2(1-\alpha)}}{xa^4}\int_{0}^{1}\frac{(1-x^{-2}s^2)^{\alpha}(1-k^{-2}x^{-2}s^2)^{1-\alpha}}{(1+a^2x^{-2}s^2)^2}\,ds
\\
=
\frac{k^{2(1-\alpha)}}{2xa^4}F^{(3)}_D(1;-\alpha,\alpha-1,2;2;
x^{-2},k^{-2}x^{-2},a^{-2}x^{-2}).
\end{multline*}

Comparing the interior and exterior moduli we immediately obtain
the following statement.

\begin{thm}\label{Qq}
Let $L$ be the isosceles trapezoidal curve with acute
angle $\pi\alpha$ and bases $c$ and $d$. Let $\lambda$ and $k$ be
solutions of \eqref{dpluseq1} and \eqref{k}, \eqref{ak}. Then the
coefficient $M_L$  satisfies the inequality
$$
M_L\ge
\max\left[\frac{{\K}(\lambda){\K}(k')}{{\K}(\lambda')(k)},\frac{{\K}(\lambda'){\K}(k)}{{\K}(\lambda)(k')}\right]
$$
where $\lambda'=\sqrt{1-\lambda^2}$, $k'=\sqrt{1-k^2}$.
\end{thm}

It is evident that the estimation from Theorem~\ref{Qq} is not
sharp. Numerical experiments  with sufficiently long rectangles
show that the ratio of moduli of two quadrilaterals, external and
internal, with the same vertices $z_1$, $z_2$, $z_3$, and $z_4$ on
the line is not maximal if  $z_k$ coincides with the 'natural'
vertices of the rectangle. The best result is for the case where
$z_k$ are on the bigger sides and are symmetric with respect to
the axes of symmetry of the rectangle. In the next section we will
try to explain this fact theoretically.

\section{The boundary of a rectangle}

Consider the case $\alpha=1/2$. Let $\Pi_d=[-d,d]\times[-1,0]$ be
a rectangle. Denote by $\Pi^c_d$ its exterior. Now we fix a number
$\delta\in(0,d]$. Let $\bm{\Pi}_\delta$ be the quadrilateral
$(\Pi_d,\delta-i,\delta,-\delta,-\delta-i,)$ and
$\bm{\Pi}_\delta^c=(\Pi_d^c,-\delta-i,-\delta,\delta,\delta-i)$.
We will compare their conformal moduli. Let ${\K}(\lambda)$ be the
complete elliptic integral of the first kind defined by
\eqref{ek}; we will write for short
${\K}'(\lambda)={\K}(\lambda')$ where
$\lambda'=\sqrt{1-\lambda^2}$.

\begin{thm}\label{kkpr}
{We have
\begin{equation}\label{0}
\frac{\text{\rm Mod}(\bm{\Pi}_\delta)}{\text{\rm
Mod}(\bm{\Pi}_\delta^c)}\ge \frac{2\lambda
{\K}'(\lambda)}{{\K}(\lambda)}\,d,
\end{equation}
where $\lambda=\delta/d$.}
\end{thm}

\begin{proof} It is obvious that
\begin{equation}\label{10}
\text{\rm Mod}(\bm{\Pi}_\delta)=2\delta.
\end{equation}
Now we will estimate $\text{\rm Mod}(\bm{\Pi}_\delta^c)$.

\begin{figure}[ht] \centering
\includegraphics[width=3.6 in]{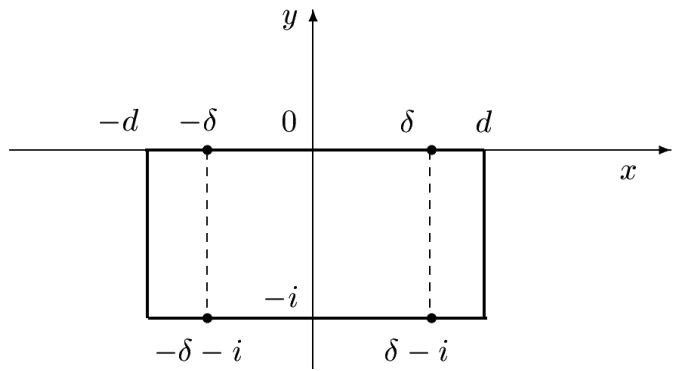}
\caption{Rectangle with
shifted vertices.}\label{fig3c}
\end{figure}

Let
${\Pi}_d^{c+}$ be the part of ${\Pi}_d^c$, lying in the quarter of
the plane $\{z:\Re z\ge 0, \Im z\ge-1/2\}$. Consider the
quadrilateral $\bm{\Pi}_\delta^{c+}:=(\Pi^{c+}_d;
0,\delta,d-i/2,\infty)$. By the symmetry principle,
\begin{equation}\label{2}
\text{\rm Mod}(\bm{\Pi}_\delta^{c+})=\text{\rm
Mod}^{-1}(\bm{\Pi}_\delta^c).
\end{equation}
On the other side, the modulus is equal to the extremal length of
the family of curves, $\Gamma$, connecting  in ${\Pi}_d^{c+}$ the
sides $[0,\delta]$ and $[d-i/2,\infty]$. Now consider the
subdomain $G$ of ${\Pi}_d^{c+}$ which is the first quarter of the
plane.
The modulus of the quadrilateral $\bm{G}:=(G;0, \delta, d,\infty)$
is equal to the extremal length of the family of curves,
$\Gamma_1$, connecting in $G$ the sides $[0,\delta]$ and
$[d,\infty]$. Since $\Gamma_1<\Gamma$, we obtain
\begin{equation}\label{3}
\text{\rm Mod}(\bm{\Pi}_\delta^{c+})\ge\text{\rm Mod}(\bm{G}).
\end{equation}
Under the conformal automorphism $w=(1/\delta)z$ of $G$, the
vertices of the quadrilateral $\bm{G}$ are mapped to the points
$0$, $1$, $1/\lambda$, $\infty$. The modulus of the obtained
quadrilateral is well-known; it is expressed via elliptic
integrals. Since conformal modulus is invariant under similarity
mappings, we find
\begin{equation}\label{4}
\text{\rm
Mod}(\bm{G})=\frac{{\K}'(\lambda)}{{\K}(\lambda)}.
\end{equation}
From \eqref{10}--\eqref{4} we obtain \eqref{0}.
\end{proof}

Consider the function
$$
g(\lambda):= \frac{\lambda
{\K}'(\lambda)}{{\K}(\lambda)}.
$$

The graph of $g(\lambda)$ is given on the Fig.~\ref{grg}. It is
evident that $\lim_{\lambda\to 0+} g(\lambda)=0$. \vskip 0.2cm

\begin{figure}[ht] \centering
\includegraphics[width=4.2 in]{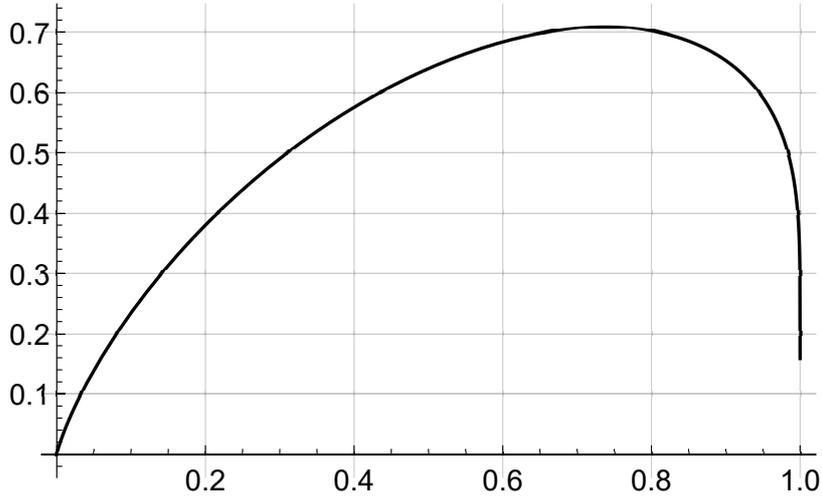}
\caption{The graph of the function $g(\lambda)$.}\label{grg}
\end{figure}

\begin{lem}\label{g} The function $g$ is concave on $(0,1)$ and has
a unique maximum point $\lambda_0=0.7373921\ldots$ which is a
unique root of the equation
\begin{equation}\label{el0}
(\lambda')^2{\K}(\lambda){\K}'(\lambda)=\pi/2,
\end{equation}
$\lambda'=\sqrt{1-\lambda^2}$, on the interval $(0,1)$. The
maximal value of $g$ is equal to $g(\lambda_0)=0.708434\ldots$
\end{lem}

\begin{proof}
In terms of the function $\mu$ we see that $g(\lambda)=2 \lambda \mu(\lambda)/\pi$ and hence
$$
\frac{\pi}{2}g'(\lambda)= \mu(\lambda)+ \lambda  \mu'(\lambda), \quad
\frac{\pi}{2}g''(\lambda)=2\mu'(\lambda)+ \lambda  \mu''(\lambda).
$$
By \eqref{mudiff} we obtain
\begin{equation}\label{primeg}
g'(\lambda)=\frac{g(\lambda)}{\lambda}\,-\frac{\pi/2}{(\lambda')^2{\K}^2(\lambda)}\,=\frac{(\lambda')^2{\K}(\lambda){\K}'(\lambda)-\pi/2}{(\lambda')^2{\K}^2(\lambda)}\,,
\end{equation}
and also, after simplification,
\[
g''(\lambda)=-\,\frac{\pi/2}{\lambda(1-\lambda^2)^2{\K}^3(\lambda)}((3-\lambda^2){\K}(\lambda)-2
{\E}(\lambda))<0.
\]
For the last inequality note that
$(3-\lambda^2){\K}(\lambda)>2{\K}(\lambda)>2{\E}(\lambda)$,
$0<\lambda<1$. Therefore, $g$ is concave on $(0,1)$ and the
statement on the maximum follows from \eqref{primeg}.
\end{proof}

\begin{cor}\label{chord}
For $\lambda\in(0,\lambda_0)$ we have
$g(\lambda)>g(\lambda_0)\lambda$.
\end{cor}

\begin{cor}\label{rect} For the coefficient of
quasiconformal reflection $M_{\partial \Pi_d}$
we have the estimation $M_{\partial \Pi_d}\ge \gamma d$ where
$\gamma=2g(\lambda_0)=1.4168687\ldots$
\end{cor}

Actually, it is easy to show that the maximal value of the function $ g(\lambda)$ is attained at a
unique point $\lambda_0 \in (0,1)$ which is the
unique root of the equation \eqref{el0} on $(0,1)$.
\vskip 0.3 cm

\begin{rem}\label{rect1}  From Corollary~\ref{rect} and \eqref{mqr} it follows that $QR_{\partial \Pi_d}\ge \gamma d$.
A sharper estimate for $QR_{\Pi_d}$  follows  from  \eqref{piest}:
$$
QR_{\partial \Pi_d}\ge (\pi/2)d, \quad \pi/2=1.5707963\ldots
$$
But the method of the proof of Theorem~\ref{kkpr}, which is rather
simple, can be used to obtain a similar estimate for the case of
isosceles trapezoids.
\end{rem}

\section{Estimation of the coefficient $M_L$ for isosceles
trapezoidal polygon $L$}

Now we apply the same method to obtain a lower  estimate for the
coefficients $M_L$ and $QR_L$ of arbitrary isosceles trapezoidal
polygon $L$.

First we will prove

\begin{lem}\label{twom}
Let $D^+$ be the part of $T^-$ lying in the half-plane $\{y\ge
-1\}$ and let $0<\delta<c$. Denote by $M_1$ the conformal modulus
of $(D;-\delta,\delta,d-i,-d-i)$ and by $M_2$ the conformal
modulus
 of the quadrilateral which is the upper half-plane with
vertices $-\delta$, $\delta$, $d$, $-d$. Then
\begin{equation}\label{m12}
M_1\ge C(\alpha) M_2
\end{equation}
where
$$
C(\alpha)=\left(\sqrt{1+\tan^2(\pi\alpha)/4}-\tan(\pi\alpha)/2\right)^2,\quad
0< \alpha\le 1/2.
$$
\end{lem}

\begin{proof} Consider the piecewise-linear mapping $F(x,y)=x+iv(x,y)$
where
$$
v(x,y)=\left\{
\begin{array}{cr}
y,& |x|\le c.\\
y+ (|x|-c)\tan(\pi\alpha),& c\le|x|\le d,\\
y+ 1,& |x|\ge d.\\
\end{array} \right.
$$
It is easy to verify that $F$ is a homeomorphism of $\mathbb{C}$,
mapping $(D;-\delta,\delta,d-i,-d-i)$ onto $(H;-\delta, \delta,
d,-d)$. We will show that $F$ is a $K$-quasiconformal mapping with
\begin{equation}\label{KK}
K=(\sqrt{4+\tan^2(\pi\alpha)}+\tan(\pi\alpha))^2/4.
\end{equation}
Actually, $F$ is conformal in $\{|x|\le c\}$ and $\{|x|\ge d\}$.
Since $F(x,y)$ is even with respect to $x$, we only need to show
that  $F$  is $K$-quasiconformal in the strip $\{c\le x\le d\}$
where it has the form
$$
F(z)=z+i(x-c)\tan(\pi\alpha)=\frac{1}{2}\,[(2+i\tan(\pi\alpha))z-i\tan(\pi\alpha)\overline{z}]-ic\tan(\pi\alpha),
$$
therefore,
$$
\left|\frac{F_{\overline{z}}}{F_z}\right|\le k:=
\frac{\tan(\pi\alpha)}{\sqrt{4+\tan^2(\pi\alpha)}}.
$$
and this implies that $F$ is $\frac{1+k}{1-k}$\,-quasiconformal
mapping. But $\frac{1+k}{1-k}=K$ where $K$ is given by \eqref{KK}.

At last, because of quasiinvariance of conformal modulus under
quasiconformal mapping, we obtain that $M_1\ge K^{-1} M_2$ where
$K^{-1}=(\sqrt{4+\tan^2(\pi\alpha)}-\tan(\pi\alpha))^2/4=C(\alpha)$.
\end{proof}
\medskip

Now, with the help of Lemma~\ref{twom} we will estimate $M_{L}$.

\begin{thm}\label{cd0}
Let $L$ be the isosceles trapezoidal polygon with acute angle
$\pi\alpha$ and bases $c$ and $d$, $c<d$.

1) If $\frac{c}{d}\ge \lambda_0$, where $\lambda_0$ is described
in Lemma~\ref{g}, then
\begin{equation}\label{ml}
M_L\ge g(\lambda_0)(1+C(\alpha))d.
\end{equation}

2) If $\frac{c}{d}<\lambda_0$, then
$$
M_L\ge g(\lambda)(1+C(\alpha))d, \quad \lambda=c/d.
$$
\end{thm}

\begin{proof} 1) Let $\frac{c}{d}\ge \lambda_0$. Then we put
$\delta=\lambda_0d$, $0<\delta<c$. Denote
$$\mbox{\boldmath$T$}_\delta^+=(T^+;\delta-i,\delta,-\delta,-\delta-i),\quad
\mbox{\boldmath$T$}_\delta^-=(T^-;-\delta-i,-\delta,\delta,\delta-i),$$
then
\begin{equation}\label{tpl}
\text{Mod}(\mbox{\boldmath$T$}_\delta^+)\ge 2\delta=2\lambda_0 d.
\end{equation}

The line $y=-1$ separates $T^-$ into two domains. One of them is
$D^+$, denote by $D^-$ the other one. Let
 $\mbox{\boldmath$D$}^+=(D^+;-\delta,\delta,d-i,-d-i)$,  $\mbox{\boldmath$D$}^-=(D^-;-\delta-i,\delta-i,d-i,-d-i)$.
Then, with the use of Lemma~\ref{twom}, we obtain
$$
\text{Mod}(\mbox{\boldmath$D$}^-)=\frac{{\K}'(\lambda_0)}{2{\K}(\lambda_0)},
\quad \text{Mod}(\mbox{\boldmath$D$}^+)\ge C(\alpha)
\text{Mod}(\mbox{\boldmath$D$}^-)
$$

Now we note that
$\mbox{\boldmath$D$}^+$, $\mbox{\boldmath$D$}^-$ and the
quadrilateral
$(\mbox{\boldmath$T$}_\delta^-)^*=(T^-;-\delta,\delta,d-i,-d-i)$,
conjugate to $\mbox{\boldmath$T$}_\delta^-$,  are symmetric with
respect to the imaginary axis.  Applying the Gr\"otzsch' lemma
\cite[pp.13-15]{vas} to the right halves of the considered
quadrilaterals and taking into mind that, by the symmetry
principle,  their moduli are  half as small as the initial moduli,
we obtain
\begin{equation}\label{tmns}
(\text{Mod}(\mbox{\boldmath$T$}_\delta^-))^{-1}=\text{Mod}((\mbox{\boldmath$T$}_\delta^-)^*)
\ge\text{Mod}(\mbox{\boldmath$D$}^+)
+\text{Mod}(\mbox{\boldmath$D$}^-)
\end{equation}
Then, multiplying \eqref{tpl} and \eqref{tmns}, we obtain
$$
\frac{\text{Mod}(\mbox{\boldmath$T$}_\delta^+)}{\text{Mod}(\mbox{\boldmath$T$}_\delta^-)}\ge
g(\lambda_0)(1+C(\alpha))d.
$$
and this implies \eqref{ml}.\vskip 0.2 cm

2) Now let $\frac{c}{d}<\lambda_0$. Then, by similar reasoning as
above, we obtain
$$
\frac{\text{Mod}(\mbox{\boldmath$T$}_\delta^+)}{\text{Mod}(\mbox{\boldmath$T$}_\delta^-)}\ge
g(\lambda)(1+C(\alpha))d
$$
where $\lambda=\frac{c}{d}$.
\end{proof}
 \vskip 0.4 cm

\begin{rem}\label{al} The estimation obtained in Theorem~\ref{cd0}
is good  for small $\alpha$ because then $C(\alpha)$  is close to
$1$. For $\alpha$, close to $\pi/2$, the value of $C(\alpha)$ is
close to zero and probably can be essentially improved.
\end{rem}

Using \eqref{mqr}, we can estimate $QR_L$ for trapezoidal curves.

\begin{cor}\label{ml1}
Under the assumption of Theorem~\ref{cd0}, if $\frac{c}{d}\ge
\lambda_0$, then $QR_L\ge g(\lambda_0)(1+C(\alpha))d$. If
$\frac{c}{d}<\lambda_0$, then $ QR_L\ge g(\lambda)(1+C(\alpha))d$
where $\lambda=c/d$.
\end{cor}

We can also enhance the estimations  given in Theorem~\ref{cd0}
and Corollary~\ref{ml1}, if we calculate
$\text{Mod}(\mbox{\boldmath$T$}_\delta^+)$ and
$\text{Mod}(\mbox{\boldmath$T$}_\delta^+)$ with the help of the
formulas given in Section~\ref{conf}. \vskip 0.2 cm

Now we will describe the algorithm in more detail.\vskip 0.2 cm

1) Let $\frac{c}{d}\ge \lambda_0$. First we find the preimages of
$\delta$ and $\delta-i$, $\delta=\lambda_0d$, under the conformal
mapping of the lower half-plane onto $T^+$ described in
Subsection~\ref{intt}.

For this, we find a unique $x_*\in[0,1]$ from the equation
$$
x_*F_1({\textstyle\frac{1}{2}};\alpha,1-\alpha;{\textstyle\frac{3}{2}};x_*^2,\lambda^2x_*^2)=\frac{\lambda_0d}{c}\,B(1-\alpha,\textstyle\frac{1}{2})\,{}_2F_1(\textstyle\frac{1}{2},1-\alpha,\textstyle\frac{3}{2}-\alpha,\lambda^2).
$$
and a unique $x_{**}\in (1/\lambda, +\infty)$,  satisfying
$$
x_{**}^{-1}\lambda^{2(1-\alpha)}F_1({\textstyle\frac{1}{2}};\alpha,1-\alpha;{\textstyle\frac{3}{2}};x_{**}^{-2},\lambda^{-2}x_{**}^{-2})=\frac{\lambda_0d}{c}\,B(1-\alpha,\textstyle\frac{1}{2})\,{}_2F_1(\textstyle\frac{1}{2},1-\alpha,\textstyle\frac{3}{2}-\alpha,\lambda^2).
$$
Now we find $\widetilde{\lambda}= x_{*}/x_{**}$ and the modulus of
$\mbox{\boldmath$T$}^+$:

\begin{equation}\label{plusen}
\text{Mod}(\mbox{\boldmath$T$}^+)=\frac{2{\K}(\widetilde{\lambda})}{{\K}(\widetilde{\lambda}')},
\quad \widetilde{\lambda}'=\sqrt{1-\widetilde{\lambda}^2}.
\end{equation}

Then we find the preimages of $\delta$ and $\delta-i$ under the
conformal mapping of the upper half-plane onto $T^-$ described in
Subsection \ref{extt}.

We find $y_*\in(0,1)$ from the equation
$$
y_*F_D^{(3)}({\textstyle\frac{1}{2}};-\alpha,\alpha-1,2;{\textstyle\frac{3}{2}};y_*^2,y_*^2k^2,-y_*^2a^2)=
\frac{\lambda_0d}{c}F_1({\textstyle\frac{1}{2}};\alpha-1,{\textstyle\frac{3}{2}}+\alpha;k^2m-a^2),
$$
and $y_{**}\in(1/k,+\infty)$ from the equation
$$
y_{**}^{-1}k^{2(1-\alpha)}a^{-4}F_D^{(3)}({\textstyle\frac{1}{2}};-\alpha,\alpha-1,2;{\textstyle\frac{3}{2}};y_{**}^{-2},y_{**}^{-2}k^{-2},-y_{**}^{-2}a^{-2}).
$$

Then we put $\widetilde{k}=y_{*}/y_{**}$ and obtain
\begin{equation}\label{minusen}
\text{Mod}(\mbox{\boldmath$T$}^-)=\frac{2{\K}(\widetilde{k})}{{\K}(\widetilde{k}')},
\quad \widetilde{k}'=\sqrt{1-\widetilde{k}^2}.
\end{equation}

From \eqref{plusen} and \eqref{minusen} we deduce that
$$
M_L\ge
\frac{{\K}(\widetilde{\lambda}){\K}(\widetilde{k}')}{{\K}(\widetilde{\lambda}'){\K}(\widetilde{k})}\,.
$$
\vskip 0.3 cm

2) If $\frac{c}{d}<\lambda_0$, then we replace $x_*$ and $y_*$ with
$1$ and in the second equation we put $c/d$ instead of
$\lambda_0$. \vskip 0.3 cm

\begin{rem} If the base $c$ is sufficiently large, then the  choice of
quadrilateral with vertices $\pm \delta$, $\pm \delta-i$ for
estimation of $M_L$ and $QR_L$ is rather good because of a result
by W.~Hayman (see, e.g. \cite[thrm.~2.3.8]{papa}). It states that
if two sides of a quadrilateral are segments on the vertical lines
$\{x=0\}$ and $\{x=1\}$ and the other two are graphs of two
continuous on $[0,1]$ functions $y=\varphi(x)$ and $y=\psi(x)$,
and $h:=\min\psi-\max\varphi>0$, then the modulus $M$ of the
quadrilateral satisfies the inequality $$h\le M\le h+1.
$$
In the symmetric case, the estimation could be improved.
\end{rem}

\section{Upper estimate of $QR_L$ for isosceles trapezoidal polygons $L$}\label{sec:ss}

Let $\alpha$ be a real number with $0<\alpha<1.$ An analytic
function $f$ on the unit disk $\mathbb{D}$ is said to be strongly
starlike of order $\alpha$ if $f'(0)\ne 0$ and if $f$ satisfies
the inequality
$$
\left|\arg\frac{zf'(z)}{f(z)-f(0)}\right|<\frac{\pi\alpha}2
$$
for $0<|z|<1.$
In particular, $\Re[zf'(z)/(f(z)-f(0))]>0$ and thus $f$ is a starlike univalent function
on the unit disk.
A simply connected domain $\Omega$ in $\mathbb{C}$
is said to be strongly starlike of order $\alpha$ with
respect to $w_0$ if the conformal homeomorphism $f:\mathbb{D}\to\Omega$
with $f(0)=w_0$ and $f'(0)>0$ is strongly starlike of order $\alpha.$
The following result is due to Fait, Krzy\. z and Zygmunt \cite{FKZ76}.
Here and hereafter, we set
\begin{equation}\label{eq:Ka}
K(\alpha)=\frac{1+\sin(\pi\alpha/2)}{1-\sin(\pi\alpha/2)}
\end{equation}
for $0<\alpha<1.$

\begin{lem}\label{lem:FKZ}
Let $0<\alpha<1.$
A strongly starlike function $f$ of order $\alpha$ on $\mathbb{D}$
extends to a $K(\alpha)$-quasiconformal mapping of $\mathbb{C}.$
\end{lem}

In particular, we see that the boundary of a strongly starlike domain is
a Jordan curve in $\mathbb{C}.$
As a consequence, we obtain the following result.

\begin{cor}
Let $L$ be the boundary of a strongly starlike domain $\Omega$ of order $\alpha.$
Then $QR_L\le K(\alpha),$ where $K(\alpha)$ is given in \eqref{eq:Ka}.
\end{cor}

\begin{proof}
We follow Ahlfors' construction \cite{ahlfors:qc2}. Let $j$ be the
inversion in the unit circle, $j(z)=1/\overline{z}.$ By Lemma
\ref{lem:FKZ}, a conformal mapping $f:\mathbb{D}\to \Omega$
extends to a $K(\alpha)$-quasiconformal mapping of the Riemann
sphere $\mathbb{C}\cup\{\infty\},$ which is denoted by the same
symbol $f.$ Then $F=f\circ j\circ f^{(-1)}$ is a
$K(\alpha)$-quasiconformal reflection across $L.$
\end{proof}

To check strong starlikeness, it is convenient to look at the following quantity.
For a domain $\Omega$ and $w_0\in\Omega,$ we define
$$
R_{\Omega,w_0}(\theta)=\sup\{ r>0: w_0+te^{i\theta}\in\Omega
~\text{for all}~t\in[0,r)\}
$$
for $\theta\in\mathbb{R}.$
The following result is contained in \cite{SugawaDual}.

\begin{lem}\label{lem:R}
Let $\Omega$ be a domain in $\mathbb{C}$ containing a point $w_0$
and let $0<\alpha<1.$
The domain $\Omega$ is strongly starlike of order $\alpha$ with respect to $w_0$
if and only if $R(\theta)=R_{\Omega,w_0}(\theta)$ is absolutely continuous and
satisfies the inequality $|R'(\theta)|/R(\theta)\le \tan(\alpha\pi/2)$
for almost all $\theta\in\mathbb{R}.$
\end{lem}

We consider the isosceles trapezoid $L$ described in
Section~\ref{conf}. Let $\Omega$ be the domain bounded by $L.$ We
now show the following for $\Omega.$

\begin{lem}\label{lem:alpha}
Let $0<s<1.$
Then $\Omega$ is strongly starlike of order $\alpha(s)$ with respect $-is,$ where
$\alpha(s)\in(0,1)$ is determined by
$$
\tan\frac{\pi\alpha(s)}{2}=\max\left\{\frac cs\,, \,\frac
d{1-s}\,, \,\frac{1-s+(d-c)d}{c+(d-c)s}\,,\,\,
\frac{s-(d-c)c}{c+(d-c)s}\right\}.
$$
\end{lem}

\begin{proof}
Since $\Omega$ is symmetric in the imaginary axis, it is enough to consider
the function $R(\theta)=R_{\Omega, -is}(\theta)$ for $-\pi/2<\theta<\pi/2.$
We define $\theta_1$ and $\theta_2$ in $(0,\pi/2)$ by requiring
$$
\tan\theta_1=\frac sc\,, \quad \tan\theta_2=\frac {1-s}d\,.
$$
Then the function $R(\theta)$ is described by
$$
R(\theta)=
\begin{cases} \medskip
\dfrac{s}{\sin\theta}\,, & \theta_1<\theta<\pi/2, \\ \medskip
\dfrac{(1-s)c+sd}{\cos\theta+(d-c)\sin\theta}\,, & -\theta_2\le\theta\le\theta_1, \\
\dfrac{1-s}{-\sin\theta}\,, & -\pi/2<\theta<-\theta_2.
\end{cases}
$$
In the first case $\theta_1<\theta<\pi/2,$ we have
$|R'(\theta)|/R(\theta)=1/\tan\theta\le 1/\tan\theta_1=c/s.$
Similarly, we have $|R'(\theta)|/R(\theta)|\le 1/\tan\theta_2=d/(1-s)$
in the third case.
When $-\theta_2<\theta<\theta_1,$ we have
$$
\frac{R'(\theta)}{R(\theta)}=
\frac{\sin\theta-(d-c)\cos\theta}{\cos\theta+(d-c)\sin\theta}\,,
\quad \left(\frac{R'(\theta)}{R(\theta)}\right)'=
\frac{1+(d-c)^2}{(\cos\theta+(d-c)\sin\theta)^2}>0.
$$
Hence, $R'/R$ is increasing in this interval and, in particular,
$$
-\frac{1-s+(d-c)d}{d-(d-c)(1-s)}=-\frac{1-s+(d-c)d}{c+(d-c)s}
\le \frac{R'(\theta)}{R(\theta)}\le \frac{s-(d-c)c}{c+(d-c)s}
$$
for $\theta\in(-\theta_2,\theta_1).$
Therefore, by Lemma \ref{lem:R}, the required formula follows.
\end{proof}

We are now able to show the following result.

\begin{thm}\label{thm:upper}
Let $L$ be an isosceles trapezoidal polygon of height $1$ and bases $c$ and $d$ with
$c\le d.$ Then
$$
QR_L\le (\sqrt{1+\tau^2}+\tau)^2,
$$
where
$$
\tau=\max\left\{c+d, \, \,\frac{1-c^2+d^2}{2c}\right\}.
$$
\end{thm}

\begin{proof}
Let $\Omega$ be the domain bounded by $L.$ If $s=c/(c+d)$, then
$c/s=d/(1-s)$ and we apply the previous lemma to show that
$\Omega$ is strongly starlike of order $\alpha$ with
$\tan(\alpha\pi/2)=\tau.$ Thus the required assertion follows with
the help of the formula
$\sin\theta=\tan\theta/\sqrt{1+\tan^2\theta}.$
\end{proof}

\begin{rem}
When $c=d,$ we have a rectangle $L$ of height 1 and width
$a=2d.$ We may assume that $a\ge 1.$ Then we have $\tau=a$ and
thus
$$
QR_L\le (\sqrt{1+a^2}+a)^2\le (3+2\sqrt 2)a^2.
$$
We recall that Werner's estimation \eqref{piest} gives us
$QR_L\le \pi a.$ Therefore, the last theorem yields only a poor
estimate.
\end{rem}

\section{Appendix~A. The Wolfram Mathematica code for calculation of exterior modulus for polygonal quadrilateral}

Here we give a code which defines the function ExtMod($A,B,n,wp$)
described in Section~\ref{conf1}.\medskip

\small

\begin{verbatim}
ExtMod[A_?NumberQ, B_?NumberQ, n_, wp_] := Module[{m = n, A1 = 1., A2 = 0.,
A3 = A, A4 = B, a, b, c, d, i, r1, r2, s, t, t1, t2, x, x0, x1, x2, x3, x4,
y0, z0, A0, B0, C0, D0, sol, J1, J2, K, L, L2, M, S, T2},
T2 = 10^m; S = IntegerPart[5 (m + 15)]; a = Arg[A4 - A1]/Pi;
b = 1 - Arg[A3 - A2]/Pi; c = 1 - b - Arg[A4 - A3]/Pi; d = 2 - a - b - c;
L = Abs[A3 - A2]; t1 = 1; t2 = T2;
   Do[t = (t1 + t2)/2; K = a + b + c - 1;
   A0 = 2 (K - 1)^2; B0 = (K - 1) (4 - 3 (a + c) + (4 - 3 (a + b)) t);
   C0 = 2 - 3 (a + c) + (a + c)^2 + 2 (3 - 5*a - 2*b - 2*c + 2*a^2 + 2*a*b
   + 2*a*c + b*c) t + (2 - 3 (a + b) + (a + b)^2) t^2;
   D0 = (1 - a) (a + c - 1 + (a + b - 1) t) t;
   sol = Solve[A0*x^3 + B0*x^2 + C0*x + D0 == 0, x];
   x1 = x /. sol[[1]]; x2 = x /. sol[[2]]; x3 = x /. sol[[3]];
   r1[x_] = a*t*x; r2[x_] = ((1 - K) x + K (t + 1) - c*t - b);
   x4 = If[r2[x1]^2 x1^2 < r1[x1] r2[x1], x1, x2];
   x0 = If[r2[x3]^2 x3^2 < r1[x3] r2[x3], x3, x4];
   y0 = Sqrt[r1[x0]/r2[x0] - x0^2]; z0 = x0 + I*y0;
   Quiet[J1 = Re[NIntegrate[s^a (1 - s)^b (1 - s/t)^c/Abs[1 - s/z0]^4,
   {s, 0, 1}, WorkingPrecision -> wp]]];
   Quiet[J2 = Re[NIntegrate[s^a (s - 1)^b (1 - s/t)^c/Abs[1 - s/z0]^4,
   {s, 1, t}, WorkingPrecision -> wp]]];
   L2 = J2/J1; If[L2 < L, t1 = t, t2 = t], {i, S}];
t = N[t]; M = EllipticK[(1. - 1/t)]/EllipticK[1/t]; {M, a, b, c, d, t, z0}];
\end{verbatim}

\normalsize
\medskip

Now we will give the results of calculating of the exterior moduli of the polygons from Example~\ref{ex1}.
The working precision is equal $16$ and the values of $n$ is equal $2$ because the values of the moduli are sufficiently small.
\medskip

\textbf{Polygon~1.}

\small
\begin{verbatim}
A = 28/25 + I*69/50; B = -19/25 + I*21/25; sol = ExtMod[B, A, 2, 16];
Print["ExtMod(", A, ", ", B, ",", 16, ") = " NumberForm[sol[[1]], 16], ","];
Print["alpha = " NumberForm[sol[[2]], 16], "," " beta = "
NumberForm[sol[[3]], 16], "," " gamma = " NumberForm[sol[[4]], 16], "," "
delta = " NumberForm[sol[[5]], 16], ","]; Print ["t = " NumberForm[sol[[6]],
16], "," " z0 = " NumberForm[sol[[7]], 16], "."]

ExtMod(28/25+(69 I)/50, -(19/25)+(21 I)/25,16) =  0.992341633097864,
alpha =  0.4723903292882761,  beta =  0.2659022512561763,
gamma =  0.6450651518079917,  delta =  0.6166422676475559,
t =  1.966910456214164,  z0 =  (1.215406699779183+1.315084271771535 I).
\end{verbatim}
\medskip

\normalsize
\textbf{Polygon~2.}

\small
\begin{verbatim}
A = 42/25 + I*4; B = -3/25 + I*21/25; sol = ExtMod[B, A, 2, 16];
Print["ExtMod(", A, ", ", B, ",", 16, ") = " NumberForm[sol[[1]], 16], ","];
Print["alpha = " NumberForm[sol[[2]], 16], "," " beta =
" NumberForm[sol[[3]], 16], "," " gamma = " NumberForm[sol[[4]], 16], "," "
delta = " NumberForm[sol[[5]], 16], ","]; Print ["t = " NumberForm[sol[[6]],
16], "," " z0 = " NumberForm[sol[[7]], 16], "."]

ExtMod(42/25+4 I, -(3/25)+(21 I)/25,16) =  0.959257171919002,
alpha =  0.4463997482438991,  beta =  0.4548327646991334,
gamma =  0.2099823197839025, delta =  0.888785167273065,
t =  1.83346758954612,  z0 =  (0.7429152683728336+1.983082728044083 I).
\end{verbatim}

\normalsize

\bigskip
\begin{table}[ht]
\caption{The values of exterior moduli for some polygonal
quadrilaterals.}
\centering
\begin{tabular}{|l|l|l|l|}
 \hline
 A & B & ExtMod[B,A] \\
 \hline
 7 + 5 I &$ -1 + 2$ I  &1.158095606321043  \\
 \hline
 8 + 3 I &  $-1$ + I & 1.130410084465672  \\
 \hline
 5 + 5 I &  $-3$ + I &  1.233703270301942  \\
  \hline
  7 + 4 I & $-3$ + 3 I &  1.274708414007269 \\
   \hline
   5 + 5 I &  $-1$ + 2 I &  1.140576491710462 \\
  \hline
  7 + 5 I&  \hspace{8mm}I&  1.015468479689712 \\
  \hline
  7 + 3 I & \hspace{3mm}1 + 2 I &   1.135151674872884 \\
 \hline
   4 + 5 I &  $-2$ + I & 1.157883901548636  \\
  \hline
     1 + I& \hspace{8mm}I &  0.999999999999995 \\
      \hline
\end{tabular}\label{tab3}
\end{table}

Table~\ref{tab3} was computed with the function ExtMod.

\begin{rem} {\rm There is also another method to validate the results of  the function ExtMod. We can compare the results
in the case of a rectangle with vertices $0,1,1+i*h, i*h, h>0,$ to the analytic formula given by Duren and Pfaltzgraff \cite{dp}, see also \cite{hrv1} for further bibliographic references.  By this formula, defining
$$\psi(r)= \frac{ 2(\E(r) -(1-r) \K(r))}{\E(\sqrt{1-r^2}) -r\K(\sqrt{1-r^2})},$$
$$ {\rm DP}(h)=\mu(\psi^{-1}(h))/\pi,$$
we have
$$  {\rm ExtMod}[I*h, 1 + I*h, 2, 16][[1]] = {\rm DP}(1/h).$$
For the range $h \in[0.5, 10]$ this last identity holds with approximate error $10^{-14}.$
}
\end{rem}
\bigskip

\bigskip

\bibliographystyle{siamplain}

\medskip

\end{document}